\documentclass{amsart}
\usepackage{graphicx}
\usepackage{amssymb}
\usepackage{amsfonts}
\usepackage[]{amsmath}
\usepackage[]{pstricks}
\usepackage[]{epsfig}
\usepackage[]{float}
\newpsobject{malla}{psgrid}{subgriddiv=1,griddots=10,gridlabels=6pt}
\newtheorem{theorem}{Theorem}[section]
\newtheorem{definition}[theorem]{Definition}
\newtheorem{proposition}[theorem]{Proposition}

\newtheorem{lemma}[theorem]{Lemma}
\newtheorem{remark}[theorem]{Remark}

\newcommand{\p}{\partial}

\newcommand{\cB}{{\mathcal B}}

\newcommand{\cL}{{\mathcal L}}
\newcommand{\cX}{{\mathcal X}}

\newcommand{\cS}{{\mathcal S}}

\newcommand{\R}{{\mathbb R}}
\newcommand{\Z}{{\mathbb Z}}
\newcommand{\C}{{\mathbb C}}
\newcommand{\T}{{\mathbb T}}

\newcommand{\ran}{\rangle}
\newcommand{\lan}{\langle}
\title[Periodic Nonlinear Waves Interactions for the Benney System]
{Well-posedness and Stability in the Periodic case for  the Benney System}
\author[J. Angulo]{J. Angulo $^1$}
\author[A. Corcho]{A. J. Corcho $^2$}
\author[S. Hakkaev]{S.  Hakkaev $^3$ $^\dagger$}

\thanks{ $^\dagger$ J. Angulo and A. J. Corcho
have been supported by the  research project Edital
Universal-CNPq/Brazil. S. Hakkaev was supported by
FAPESP/Brazil. The third author would like to express his
thanks to the Institute of Mathematics and Statistic (IME) at the
University of S\~ao Paulo/SP-Brazil for its hospitality} \subjclass{35Q55, 35Q60.}

\keywords{Benney system, Well-Posedness, Stability of periodic
traveling  waves.}

\begin{document}
\maketitle

{\scriptsize \centerline{$^1$Department of Mathematics, IME-USP}
 \centerline{Rua do Mat\~ao 1010, Cidade Universit\'aria, CEP 05508-090. }
 \centerline{S\~ao Paulo, SP, Brazil. \email{angulo@ime.usp.br}}}
 {\scriptsize \centerline{$^2$ Universidade Federal de Alagoas Instituto de Matem\'atica}
\centerline{Campus A. C. Sim\~oes, Tabuleiro dos Martins, 57072-900.}
\centerline{Macei\'o, AL, Brazil.
\email{adan@mat.ufal.br}}}
{\scriptsize \centerline{$^3$Faculty of Mathematics and Informatics, Shumen University.}
 \centerline{9712 Shumen, Bulgaria.  \email{shakkaev@fmi.shu-bg}}}

\setcounter{page}{1}

\begin{quote}
{\normalfont\fontsize{8}{10}\selectfont {\bfseries Abstract.} We
establish local well-posedness results in weak periodic function
spaces for the Cauchy problem of the Benney system. The Sobolev
space $H^{1/2}\times L^2$ is the lowest regularity attained and
also we cover the energy space $H^{1}\times L^2$, where global
well-posedness follows from the conservation laws of the system.
Moreover, we show the existence of smooth explicit family of
periodic travelling waves of \emph{dnoidal} type and we prove,
under certain conditions, that this family is orbitally stable in
the energy space.
\par}
\end{quote}

\section{Introduction}

In this paper we consider the system introduced by Benney in
\cite{Benney2} which models the interaction between short and long
waves, for example in the theory of resonant water wave
interaction in nonlinear medium:
\begin{equation}
\label{benney}
\begin{cases}
iu_t+u_{xx}=uv +\beta |u|^2u , &  (x,t)\in
\mathcal{M}\times \triangle T \\
v_t  = (|u|^2)_x,\\
u(x,0)=u_0(x) ,\quad v(x,0)=v_0(x),
\end{cases}
\end{equation}
where $u=u(x,t)$ is a complex valued function representing the
enveloped of short waves, and $v=v(x,t)$ is a real valued function
representing the long wave. Here $\beta $ is a real parameter,
$\triangle T$ is the time interval $[0,T]$ and $\mathcal{M}$ is
the real line $\R$ or the one dimensional torus $\T=\R/\Z $.

We let $H^s(\mathcal{M})$ by denoting the classical Sobolev space
with the norm
\begin{equation*}
\|f\|_{s}=\small{\left (\displaystyle\int_{-\infty}^{\infty}(1+|\xi|)^{2s}|\hat{f}(\xi)|^2dx\right)^{1/2}}\quad \text{if}\;\mathcal{M}=\R,
\end{equation*}
and
\begin{equation*}
\|f\|_{s}=\small{\left (\displaystyle\sum_{n\in \Z}(1+|n|)^{2s}|\hat{f}(n)|^2dx\right)^{1/2}}\quad \text{if}\;\mathcal{M}=\T,
\end{equation*}
where $\hat{f}(\xi)$ and $\hat{f}(n)$ denote  the Fourier transform and Fourier coefficient of $f$, respectively. We
consider the initial data $(u_0,v_0)$ in the space
$H^r(\mathcal{M})\times H^s(\mathcal{M})$ with the induced norm
$$\|(f,g)\|_{r\times s}:=\|f\|_r+\|g\|_s.$$

The following quantities
\begin{align}
&E_1[u(.,t)]=\int_I|u(x,t)|^2dx,\label{CL-1}\\
&E_2[u(.,t),v(.,t)]=\int_I \left[v(x,t)|u(x,t)|^2+
|u_x(x,t)|^2+\tfrac{\beta}{2}|u(x,t)|^4\right]dx\label{CL-2}\\
\intertext{and}
&E_3[u(.,t), v(.,t)]=\int_I \Bigl[
|v(x,t)|^2+ 2 \text{Im}\;(u(x,t)\bar{u}_x(x,t))
\Bigl]dx\label{CL-3}
\end{align}
with the interval $I=(-\infty, +\infty)$  if $\mathcal{M}=\R$ and $I=[0,1]$  if $\mathcal{M}=\T$
are invariants by the flux of the system (\ref{benney}); i.e, the natural energy space for the system is  $H^1(\mathcal{M})\times
L^2(\mathcal{M})$.

\subsection{Some results in the continuous case}
When $\mathcal{M}=\R$ the local well-posedness for (\ref{benney})
for data $(u_0,v_0)\in H^{(s+1/2)}(\R)\times H^s(\R)$ with indices
$s\ge 0$ was established in the works \cite {Bekiranov},
\cite{Ginibre} and \cite{Tsutsumi-Hatano}. Furthermore, in
\cite{Tsutsumi-Hatano} also was proved global well-posedness in
$H^{(s+1/2)}(\R)\times H^s(\R)$ for $s=0$ if $\beta =0$ and for
$s\in \Z^+$ and any real $\beta$ by using the conservation laws
(\ref{CL-1}), (\ref{CL-2}) and (\ref{CL-3}).

Recently, in \cite{Corcho} Corcho showed that for $\beta < 0$
(focusing case) and for data $(u_0,v_0)\in H^r(\R)\times H^s(\R)$,
with $0\le 3r+1<1$ and $r(2s+3)+1\ge 0$, this problem is ill-posed
in the following sense: the data-solution mapping fails to be
uniformly continuous on bounded sets of $H^r(\R)\times H^ s(\R)$.

Concerning to the existence and stability  of solitary waves
solutions for (\ref{benney}) of the general form
\begin{equation}\label{eq:ptw}
\left \{
    \begin{array}{l}
u(x,t)=e^{i\omega t}e^{ic(x-ct)/2}\phi_s(x-ct),\\
v(x,t)=\psi_s(x-ct),
\end{array}
   \right.
\end{equation}
where $\phi_s,\psi_s:\Bbb R\to \Bbb R$ are smooth, $c>0$,
$\omega\in \Bbb R$, and $\phi_s(\xi), \psi_s(\xi)\to 0$ as
$|\xi|\to \infty$, Lauren\c cot in \cite{Laurencot} studied for
$\beta =0$, the nonlinear stability of the orbit
$$
\Omega_{(\Phi,\Psi)}=\left\{(e^{i\theta}\Phi(\cdot+x_0),\Psi(\cdot+x_0));\;
(\theta,x_0)\in [0,2\pi)\times \Bbb R\right\},
$$
in $H^1(\Bbb R)\times L^2(\Bbb R)$ by the flow generated by
(\ref{benney}). Here we have that
$\Phi(\xi)=e^{ic\xi/2}\phi_{s}(\xi)$, $\Psi(\xi)=\psi_{s}(\xi)$,
and
\begin{equation}\label{eq:sol}
 \phi_{s}(\xi)=\sqrt{2c\sigma}\text{sech}(\sqrt
\sigma \xi),\qquad \psi_{s}(\xi)=-\frac{1}{c}\phi_{s}^2(\xi)
\end{equation}
$\sigma=\omega-\frac{c^2}{4}>0$.

\subsection{Main results in the periodic case}
In the present work  we focus the atten\-tion on the case $\mathcal{M}=\T$ and we study the following problems:
\begin{itemize}
\item well-posedness in Sobolev spaces with low regularity,
\item existence and nonlinear stability of periodic
travelling waves
\end{itemize}
for the periodic initial values $(u_0,v_0)$  belonging into the  space  $H^r(\T)\times H^s(\T)$, also
denoted by $H^r_{per}\times H^s_{per}$.

As follows we define the concepts of well-posedness and the
stability that will be use in this work.
\begin{definition}[Well-posedness and Ill-posedness]
We say that the system (\ref{benney}) is locally well-posed, in
time,  in the space   $H^r_{per}\times H^s_{per}$ if
the following conditions hold:
\begin{enumerate}
    \item [(a)] for every $(u_0,v_0)$ in the space  $H^r_{per}\times H^s_{per}$  there exists a positive time
    $T=T\left (\|u_0\|_r, \|v_0\|_s \right )$ and a  distributional solution
    $(u,v):\T\times \triangle T \longrightarrow \C \times \R$ which is in the space
    $C\left(\triangle T;\; H^r_{per}\times H^s_{per} \right)$;

    \vspace{0.15cm}
    \item [(b)] the data-solution mapping $(u_0,v_0)\longmapsto (u,v)$ is uniformly continuous
    from $H^r_{per}\times H^s_{per}$ to $C\left (\triangle T;\; H^r_{per}\times H^s_{per}\right )$;

    \vspace{0.15cm}
    \item [(c)] there is an additional Banach space $\mathcal{X}$ such that $(u,v)$ is the unique
    solution to the Cauchy problem in
    $\mathcal{X}\cap C\left (\triangle T;\; H^r_{per}\times H^s_{per}\right)$.
    \end{enumerate}
    Moreover, we say that the problem is ill-posed if, at least, one
    of the above conditions fails.
\end{definition}

Before stating our well and ill posedness results we will give some useful notations. Let
$\eta$ be a function in $C_0^{\infty}(\R)$ such that $0\le \eta(t)
\le 1$,
$$\eta(t)=
\begin{cases}
1&\;\; \text{if}\;\;|t|\le 1,\\
0&\;\; \text{if}\;\;|t|\ge 2,
\end{cases}$$
and $\eta_{\delta}(t)=\eta (\tfrac{t}{\delta})$. We denote by
$\lambda \pm$ a number slightly larger, respectively smaller, than
$\lambda$ and by $\lan \cdot \ran$, $\lan \xi \ran= 1+|\xi|$. The
characteristic function on the set $A$ is denoted by $\chi_{A}$.
Furthermore, we will work with the auxiliary periodic Bourgain
space $X^{s,b}_{per}$ defined as follows: first we denote by $\cX$
the space of functions $f: \T \times \R \rightarrow \C$ such that
\begin{enumerate}
\item [(\,i\,)] $f(x,\cdot) \in \cS (\R)$\;for each\;  $x\in \T$;
\item [(ii)] $f(\cdot,t) \in C^{\infty}(\T)$\; for each\; $t\in\R$.
\end{enumerate}
For $s, b \in \R$,  the spaces $H^{b}_tH^{s}_{per}$ and
$X^{s,b}_{per}$ are the completion of $\cX$ with respect to the
norms
\begin{equation}\label{Bourgain-Space-0}
\|f\|_{H^b_tH^{s}_{per}}=\left(\sum_{n\in \Z}
\int\limits_{-\infty}^{+\infty}(1+ |n|)^{2s}(1+ |\tau
|)^{2b}|\widehat{f}(n,\tau)|^2d\tau \right)^{\frac{1}{2}}
\end{equation}
and
\begin{equation}\begin{split}\label{Bourgain-Space}
\|f\|_{X^{s,b}_{per}}&=\|S(-t)f\|_{H^b_tH^{s}_{per}}\\
&=\left(\sum_{n\in \Z} \int\limits_{-\infty}^{+\infty}(1+
|n|)^{2s}(1+ |\tau+ n^2|)^{2b}|\widehat{f}(n,\tau)|^2d\tau
\right)^{\frac{1}{2}},
\end{split}\end{equation} respectively, where $S(t):=e^{it\partial_x^2}$ is the
corresponding Schr\"odinger generator (unitary group) associated
to the linear problem,
\begin{equation}
\begin{cases}
 iu_t+u_{xx}=0\\
  u(x,0)=g_(x).
  \end{cases}
\end{equation}
For any $r, s\in \R$\; and\;  $b_1, b_2 >1/2$, we have the
embedding $X^{r,b_1}_{per}\hookrightarrow C\left (\R;
H^r_{per}\right)$ and $H^{b_2}_tH^{s}_{per} \hookrightarrow C\left
(\R; H^s_{per}\right)$. For the case $b=1/2$ the  embedding can be
guaranteed by considering the following slightly modifications of
the Bourgain spaces:
\begin{equation}\label{Bourgain-Space-modified}
\|f\|_{X^r_{per}}:=\|f\|_{X^{r,1/2}_{per}}+ \|\lan n\ran^r\widehat
f (n,\tau)\|_{\ell^2_nL^1_{\tau}}
\end{equation}
and
\begin{equation}\label{Bourgain-Space-0-modified}
\|f\|_{Y^s_{per}}:=\|f\|_{H^{1/2}_tH^{s}_{per}}+  \|\lan
n\ran^s\widehat f (n,\tau)\|_{\ell^2_nL^1_{\tau}}
\end{equation}

Concerning local well-posedness we obtain the following result:

\begin{theorem}[Local Well-Posedness]\label{local-theorem-periodic}
For any $(u_0,v_0)\in H^r_{per}\times H^s_{per}$ with $r, s$ verifying the condition
\begin{equation}\label{local-theorem-periodic-a}
\max\{0,\,r-1\}\le s \le \min\{r,\,2r-1\},
\end{equation}
there exist a positive time $T=T\left (\|u_0\|_{r},
\|v_0\|_{s}\right)$ and a unique solution $(u(t),v(t))$ of the
initial value problem (\ref{benney}), satisfying
\begin{enumerate}
\item [(a)] $\left(\eta_T(t)u, \eta_T(t)v\right)\in X_{per}^r
\times Y_{per}^s$; \vspace{0.2cm} \item [(b)] $(u, v)\in
C\left(\triangle T;\; H^r_{per} \times H^s_{per}\right)$.
\end{enumerate}
Moreover, the map $(u_0,v_0) \longmapsto (u(t),v(t))$ is locally
uniformly continuous  from $H^r_{per} \times H^s_{per}$ into $C\left
(\triangle T;\; H^r_{per} \times H^s_{per} \right )$.
\end{theorem}

The proof of  Theorem \ref{local-theorem-periodic} is based on the
Banach fixed point theorem applied  on the integral formulation of
the system combined with new sharp periodic bilinear estimates, in
adequate mixed Bourgain spaces $X_{per}^{r,b_1} \times
H^{b_2}_tH^s_{per}$, for the coupling terms $uv$ and
$\p_x(|u|^2)$.

Also we find  a  region which the Cauchy problem is not locally
well-posed, more precisely we prove the following theorem:

\begin{theorem}\label{ill-posedness-theorem}
Let $\beta \neq 0$. Then for any $r<0$ and $s\in \R$, the initial
value problem (\ref{benney}) is locally ill-posed for data in
$H^{r}_{per}\times H^{s}_{per}$.
\end{theorem}

\begin{figure}[h]\label{Region-Periodic}
\includegraphics[width=6.9cm]{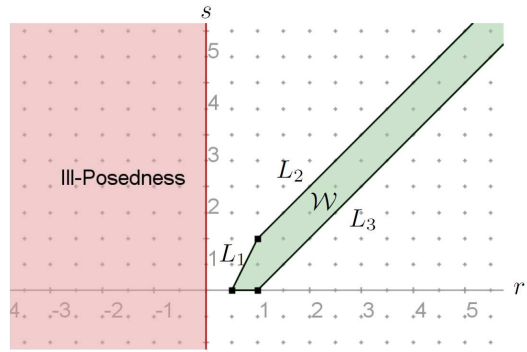}
\vspace{0.5cm} \caption{{\small Well-posedness results for periodic Benney
system. The region $\mathcal{W}$, limited by the lines $L_1:
s=2r-1$, $L_2: s=r$ and $L_3: s=r-1$, contain the indices $(r,s)$
where the local well-posedness is achieved in Theorem
\ref{local-theorem-periodic}.}}
\end{figure}

Regarding the stability of periodic travelling waves, namely,
solutions for (\ref{benney}) of the form
\begin{equation}\label{21}
\begin{cases}
         u(t,x)=e^{-i\omega
         t}e^{ic(x-ct)/2}\varphi_{\omega, c}(x-ct) \\
         v(x,t)=n_{\omega,c}(x-ct)
 \end{cases}
 \end{equation}
where $\varphi_{\omega,c}$, $n_{\omega,c}$ are real smooth,
$L$-periodic functions, $c>0$, and $\omega <0$, we have the
following definition.
\begin{definition}[Non-Linear Stability] The periodic traveling
wave $\Phi(\xi)=e^{ic\xi/2}\varphi_{\omega, c}(\xi)$,
$\Psi(\xi)=n_{\omega, c}(\xi)$, is orbitally stable in
$H^1_{per}([0,L])\times L^2_{per}([0,L])$ if for all
   $\varepsilon >0$, there exists $\delta >0$, such that if
   $||(u_0,v_0)-(\Phi,\Psi) ||_{H^1_{per}\times L^2_{per}}<\delta $ and $(u(t),v(t))$
   is the solution of (\ref{benney}) with
   $(u(0),v(0))=(u_0,v_0)$, then
     $$ \inf_{s\in [0,2\pi)} \inf_{r\in\Bbb R}
     ||(u(t),v(t))-(e^{is}\Phi(\cdot+r), \Psi(\cdot+r))||_{H^1_{per}\times L^2_{per}}<
     \varepsilon, \  \  \  \ t\in\Bbb R.
     $$
   Otherwise $(\Phi,\Psi)$ is called orbitally
   unstable.
  \end{definition}

We will show below that there exist a smooth explicit family of
profiles solutions of minimal period $L$,
$$
(\omega,c)\in \mathcal A_{\beta}\to (\varphi_{\omega, c},
n_{\omega,c}) \in H^r_{per}([0,L])\times H^s_{per}([0,L]),
$$
where $\mathcal A_{\beta}=\{(x,y): y>0, 1>\beta y, \ \ and\ \
x<-\frac{2\pi^2}{L^2}-\frac{y^2}{4}\}$ and which depends  of the
Jacobian elliptic function {\it dn} called {\it dnoidal}, more
precisely,
 \begin{equation}\label{psolu}
\begin{cases}
                \varphi_{\omega,c}(\xi)=\sqrt{\frac{c}{1-\beta c}}\;\eta_1
                dn\left( {\frac{\eta_1}{\sqrt{2}}}\xi ; \kappa
                \right) \\
                n_{\omega,c}(\xi)=-{\frac{\eta_1^2}{1-\beta c}}
                dn^2\left( {\frac{\eta_1}{\sqrt{2}}}\xi ; \kappa\right)
\end{cases}
\end{equation}
with $\eta_1=\eta_1(\omega,c)$ and $\kappa=\kappa(\omega,c)$,
being smooth functions of $\omega$ and $c$.

So, by following Angulo \cite{An} and Grillakis {\it et al.}
\cite{GrShSt1}, \cite{GrShSt2},  we obtain the following stability
theorem.

\begin{theorem}[Stability Theory]\label{stability theorem} Let
$(\omega,c)\in \mathcal A_\beta$ such that for $c>0$ there is
$q\in \mathbb N$ satisfying $4\pi q/{c}=L$. Define $\sigma\equiv
-\omega-\frac{c^2}{4}$.  Then
$\Phi(\xi)=e^{ic\xi/2}\varphi_{\omega,c}(\xi)$,
$\Psi(\xi)=n_{\omega,c}(\xi)$, with $\varphi_{\omega,c},
n_{\omega,c}$ given in (\ref{psolu}),
is orbitally stable in $H^1_{per}([0,L])\times L^2_{per}([0,L])$ by the
periodic flow generated by (\ref{benney}):
\begin{enumerate}
 \item[(a)] for $\beta \leq 0$,

 \item[(b)] for $\beta >0$ and $8\beta \sigma-3c(1-\beta c)^2\leq 0.$
 \end{enumerate}
\end{theorem}


\section{Local theory}  We prove Theorem
\ref{local-theorem-periodic} using the standard technique, that
is: we use the Duhamel integral formulation for the system
(\ref{benney}) combined with the Banach fixed point theorem in
adequate  Bourgain spaces $X_{per}^r\times Y_{per}^s$ with the
objective to get the desired solution. The main difficulty is the
necessity to prove two news mixed periodic bilinear estimates,
which we will prove in the following sections.

\subsection{\textbf{Sharp Periodic Bilinear Estimates}}
We begin recalling the following elementary inequalities, which
will be used in the proof of the next main estimates.

\begin{lemma}\label{Lemma-Calculus} Let $\theta_1, \theta_2 >0$
with  $\theta_1+\theta_2>1$ and $\lambda >1/2$. Then, there are a
positive constants $C_1$ and $C_2$ such that
\begin{enumerate}
\item [(a)] $\int\limits_{-\infty}\limits^{+\infty}\frac{dx}{\lan x-a\ran^{\theta_1}
\lan x-b\ran^{\theta_2}} \le \frac{C_1}{\lan a-b\ran^{\mu}},$
where $\mu:=\min\{\theta_1, \theta_2,\theta_1+\theta_2-1 \}$;
\vspace{0.5cm}
\item [(b)] $\sum\limits_{n\in \Z}\frac{1}{\lan n^2+an+b\ran^{\lambda}}\le C_2,$
with $a,\,b\in \R$.

\end{enumerate}
\end{lemma}
\begin{proof}
For details of the proof we refer, for instance,  the works
\cite{KPV2} and \cite{ACM}.
\end{proof}
\smallskip
\begin{lemma}\label{proposition-uv-periodic}
Let $0<\theta < 1/4$. Then, the following estimates
\begin{equation}\label{Estimativa-uv}
\|uv\|_{X_{per}^{r,-1/2}}\lesssim
\|u\|_{X_{per}^{r,1/2-\theta}}\|v\|_{H^{1/2}_t H_{per}^{s}}+
\|u\|_{X_{per}^{r,1/2}}\|v\|_{H^{1/2-\theta}_t H_{per}^{s}}
\end{equation}

\begin{equation}\label{Estimativa-uv-complement}
\left \|\lan n\ran^{r}\frac{\widehat{uv}(n,\tau)}{\lan \tau
+n^2\ran } \right \|_{\ell^2_nL^1_{\tau}}\lesssim
\|u\|_{X_{per}^{r,1/2-\theta}}\|v\|_{H^{1/2}_t H_{per}^{s}}+
\|u\|_{X_{per}^{r,1/2}}\|v\|_{H^{1/2-\theta}_t H_{per}^{s}}
\end{equation}

hold provided $r\ge 0$ and $\max\{0,\, r-1\}\le s$.
\end{lemma}
\begin{proof}

First we prove (\ref{Estimativa-uv}). We define
$f(n,\tau):=\lan\tau+n^2\ran^{b_1}\lan n\ran^r
\widehat{u}(n,\tau)$ and $g(n,\tau):= \lan\tau\ran^{b_2} \lan
n\ran^s \widehat{v}(n,\tau)$. Then, using duality arguments we
obtain
$$\|uv\|_{X_{per}^{r,-1/2}}=\sup \bigl \{W(\varphi): \|\varphi\|_{\ell^2_nL^2_{\tau}} \le 1\bigl \},$$
where,

{\small
\begin{equation}\label{e.uv-periodic-2}
W(\varphi)=\sum\limits_{(n, n_1)\in \Z^2} \int_{\R^2}
\frac{\lan\tau+n^2\ran^{-1/2} \lan n\ran^r f(n_1,\tau_1)
g(n-n_1,\tau -\tau_1) \varphi(n,\tau)}{\lan\tau_1+n_1^2\ran^{b_1}
\lan\tau -\tau_1\ran^{b_2}\lan n_1\ran^r \lan n-n_1\ran^s}d\tau
d\tau_1.
\end{equation}
}
We will divide the space $\mathbb{Z}^2\times \mathbb{R}^2$ in
three regions, namely $\mathbb{Z}^2\times \mathbb{R}^2 = A_0 \cup
A_1 \cup A_2$ and we separate the integral $W$ as follows:
\begin{equation}\label{e.uv-periodic-2}
W(\varphi)= W_0(\varphi) + W_1(\varphi)+ W_2(\varphi),
\end{equation}
where
$$W_j(\varphi)=\sum \int \sum \int \limits_{(n,n_1,\tau,\tau_1)\in A_j}
\frac{\lan\tau+n^2\ran^{-1/2} \lan n\ran^r f(n_1,\tau_1)
g(n-n_1,\tau -\tau_1) \varphi(n,\tau)}{\lan\tau_1+n_1^2\ran^{b_1}
\lan\tau -\tau_1\ran^{b_2}\lan n_1\ran^r \lan n-n_1\ran^s},$$ for
$j=0,1,2.$ It is easy to see that to obtain (\ref{Estimativa-uv})
is suffices to prove that whenever $r,\; s\ge 0$ and $r-s\le 1$
the estimate
\begin{equation}\label{e.uv-periodic-1}
W_j(\varphi)\lesssim \|f\|_{\ell^2_nL^2_{\tau}}
\|g\|_{\ell^2_nL^2_{\tau}} \|\varphi\|_{\ell^2_nL^2_{\tau}}
=\|u\|_{X_{per}^{b_1,r}}\|v\|_{H^{b_2}_t
H_{per}^{s}}\|\varphi\|_{\ell^2_nL^2_{\tau}},
\end{equation}
holds with  $b_1=1/2-\theta$ and $b_2=1/2$ or  with $b_1=1/2$ and
$b_2=1/2-\theta$. Indeed, next we will prove the following
estimates:
\begin{equation}\begin{split}\label{e.uv-periodic-1}
&W_j(\varphi)\lesssim \|u\|_{X_{per}^{1/2}}\|v\|_{H^{1/2-\theta}_t
H_{per}^{s}}\|\varphi\|_{\ell^2_nL^2_{\tau}},\; \text{for}\;
j=0,1,\\
&W_2(\varphi)\lesssim
\|u\|_{X_{per}^{1/2-\theta,r}}\|v\|_{H^{1/2}_t
H_{per}^{s}}\|\varphi\|_{\ell^2_nL^2_{\tau}}.
\end{split}\end{equation}
For this purpose, in region $A_0$ we integrate first over
$(n_1,\tau_1)$, in region $A_1$  we integrate first over
$(n,\tau)$ and in region $A_2$ we integrate first over
$(n_2,\tau_2)=(n-n_1,\tau-\tau_1)$; then using Cauchy-Schwarz
inequality we easily see that it remains only to bound uniformly
the following three expressions:

\begin{equation}\label{e.w1}
\widetilde{W}_0:=\sup\limits_{n,\tau} \frac{\lan
n\ran^{2r}}{\lan\tau+n^2\ran} \sum\limits_{n_1}\int \limits_{A_0}
\frac{d\tau_1}{ \lan\tau_1+n_1^2\ran \lan\tau_2\ran^{1-2\theta}
\lan n_1\ran^{2r} \lan n_2\ran^{2s}}
\end{equation}
\begin{equation}\label{e.w2}
\widetilde{W}_1:=\sup\limits_{n_1,\tau_1} \frac{1}{\lan
n_1\ran^{2r} \lan\tau_1+n_1^2\ran} \sum\limits_{n}\int
\limits_{A_1} \frac{\lan n\ran^{2r}d\tau}{ \lan\tau+n^2\ran
\lan\tau_2\ran^{1-2\theta} \lan n_2\ran^{2s}}
\end{equation}
\begin{equation}\label{e.w3}
\widetilde{W}_2:=\sup\limits_{n_2,\tau_2} \frac{1}{\lan
n_2\ran^{2s} \lan\tau_2\ran} \sum\limits_{n}\int \limits_{A_2}
\frac{\lan n\ran^{2r}d\tau}{ \lan\tau+n^2\ran
\lan\tau_1+n_1^2\ran^{1-2\theta} \lan n_1\ran^{2r}}
\end{equation}
Now we define the regions $A_0$, $A_1$ and $A_2$.  We use the
notation
\begin{equation}
\cL:=\max \Bigl\{\left |\tau+n^2\right|,
\left|\tau_1+n_1^2\right|, |\tau_2|\Bigl\}.
\end{equation}
and first we introduce the subsets:

{\small
\begin{equation}
\begin{split}
&A_{0,1}:=\left\{(n,n_1,\tau,\tau_1)\in\mathbb{Z}^2\times \mathbb{R}^2: |n|\le 100 \right\}, \\
&A_{0,2}:=\left\{(n,n_1,\tau,\tau_1)\in\mathbb{Z}^2\times \mathbb{R}^2: |n|> 100 \textrm{ and } |n|\le 2|n_1|\right\}, \\
&A_{0,3}:=\Bigl\{(n,n_1,\tau,\tau_1)\in\mathbb{Z}^2\times
\mathbb{R}^2: |n|> 100, |n_1|<|n|/2 \textrm{ and }
\cL=\left|\tau+n^2\right|\Bigl\}.
\end{split}
\end{equation}}
Then, we put {\small
\begin{equation}
\begin{split}
&A_0:=A_{0,1}\cup A_{0,2}\cup A_{0,3},\\
&A_1:=\Bigl \{(n,n_1,\tau,\tau_1)\in \Z^2 \times \R^2: |n|>100,
|n_1|<|n|/2 \textrm{ and } \cL=\left|\tau_1+n_1^2\right| \Bigl \},\\
&A_2:=\Bigl \{(n,n_1,\tau,\tau_1)\in\Z^2\times \R^2: |n|>100,
|n_1|<|n|/2 \textrm{ and } \cL=|\tau_2|\Bigl \}.
\end{split}
\end{equation}}

For later use, we recall that the dispersive relation of this
bilinear estimate is:
\begin{equation}\label{e.dispersive-uv-periodic}
\tau+n^2-(\tau_1+n_1^2)-\tau_2 = n^2-n_1^2,
\end{equation}
where $\tau -\tau_1=\tau_2\quad \text{and}\quad n-n_1=n_2.$

We begin with the analysis of (\ref{e.w1}). In the region
$A_{0,1}$, using that $|n|\lesssim 1$ and $r,\;s \ge0 $ we have
\begin{equation}\label{e.w1-A01}
\begin{split}
\widetilde{W}_{0,1}&:=\sup\limits_{n,\tau} \frac{\lan
n\ran^{2r}}{\lan\tau+n^2\ran} \sum\limits_{n_1}\int
\limits_{A_{0,1}}
\frac{d\tau_1}{\lan\tau_1+n_1^2\ran \lan\tau_2\ran^{1-2\theta} \lan n_1\ran^{2r} \lan n_2\ran^{2s}}\\
&\lesssim \sup \limits_{n,\tau} \frac{1}{\lan\tau+n^2\ran}\sum
\limits_{n_1}\int \limits_{-\infty}\limits^{+\infty}
\frac{d\tau_1}{\lan\tau_1+n_1^2\ran
\lan\tau_2\ran^{1-2\theta} \lan n_1\ran^{2r} \lan n_2\ran^{2s}}\\
&\lesssim \sup\limits_{\tau} \sum \limits_{n_1}
\frac{1}{\lan\tau+n_1^2\ran^{1-2\theta}} \lesssim 1,
\end{split}
\end{equation}
where in the last inequality we have used that $0< \theta < 1/4$
combined with {Lemma~\ref{Lemma-Calculus}}.

In the region $A_{0,2}$, we have that $\lan n\ran^{2r}\lesssim
\lan n_1\ran^{2r}$. Thus, similarly to the previous case, we get
\begin{equation}\label{e.w1-A02}
\begin{split}
\widetilde{W}_{0,2}&:=\sup\limits_{n,\tau} \frac{\lan
n\ran^{2r}}{\lan\tau+n^2\ran}
\sum\limits_{n_1}\int\limits_{A_{0,2}}
\frac{d\tau_1}{ \lan\tau_1+n_1^2\ran \lan\tau_2\ran^{1-2\theta} \lan n_1\ran^{2r} \lan n_2\ran^{2s}} \\
&\lesssim \sup\limits_{n,\tau} \frac{1}{\lan\tau+n^2\ran}
\sum\limits_{n_1}\int \limits_{-\infty}\limits^{+\infty}
\frac{d\tau_1}{ \lan\tau_1+n_1^2\ran
\lan\tau_2\ran^{1-2\theta} \lan n_2\ran^{2s}}\\
&\lesssim \sup\limits_{\tau} \sum\limits_{n_1}
\frac{1}{\lan\tau+n_1^2\ran^{1-2\theta}} \lesssim 1.
\end{split}
\end{equation}

In the region $A_{0,3}$ we have that $|n_1|<|n|/2$ and $|n|>100$,
which imply that $|n-n_1|\sim |n+n_1|\sim |n|.$ Moreover, the
dispersive relation~(\ref{e.dispersive-uv-periodic}) says that
$$\cL=|\tau+n^2|\gtrsim |n^2-n_1^2|=|n-n_1||n+n_1|\sim |n|^2.$$
Therefore,
\begin{equation}\label{e.w1-A03}
\begin{split}
\widetilde{W}_{0,3}&:=\sup\limits_{n,\tau} \frac{\lan
n\ran^{2r}}{\lan\tau+n^2\ran}
\sum\limits_{n_1}\int\limits_{A_{0,3}}
\frac{d\tau_1}{\lan\tau_1+n_1^2\ran \lan\tau_2\ran^{1-2\theta} \lan n_1\ran^{2r} \lan n_2\ran^{2s}} \\
&\lesssim \sup\limits_{n,\tau} \frac{\lan
n\ran^{2r-2s}}{\lan\tau+n^2\ran}
\sum\limits_{n_1}\int\limits_{-\infty}\limits^{+\infty}
\frac{d\tau_1}{\lan\tau_1+n_1^2\ran
\lan\tau_2\ran^{1-2\theta} }\\
&\lesssim \sup\limits_{n,\tau} \frac{\lan n \ran^{2r-2s}}{\lan
n\ran^2}\sum\limits_{n_1} \frac{1}{\lan\tau+n_1^2\ran^{1-2\theta}}
\lesssim 1,
\end{split}
\end{equation}
since $r \ge 0$, $r-s \le 1$ and $0<\theta < 1/4$.

Putting together the estimates (\ref{e.w1-A01}), (\ref{e.w1-A02})
and (\ref{e.w1-A03}) we conclude that
$$|\widetilde{W}_0|\le |\widetilde{W}_{0,1}|+|\widetilde{W}_{0,2}|+|\widetilde{W}_{0,3}|\lesssim 1,$$
obtaining the desired bounded  for (\ref{e.w1}).

Next we estimate the contribution of~(\ref{e.w2}). In the region
$A_1$, we know that $|n_1|<|n|/2$, $|n|>100$ and
$\cL=|\tau_1+n_1^2|$. So, $|n_2|\sim |n|$ and the dispersive
relation~(\ref{e.dispersive-uv-periodic}) implies that
$|\tau_1+n_1^2|\gtrsim n^2$. Thus,
\begin{equation*}
\begin{split}
\widetilde{W}_1&=\sup\limits_{n_1,\tau_1} \frac{1}{\lan
n_1\ran^{2r} \lan\tau_1+n_1^2\ran}
\sum\limits_{n}\int\limits_{A_1} \frac{\lan n\ran^{2r}}{
\lan\tau+n^2\ran \lan\tau_2\ran^{1-2\theta}
\lan n_2\ran^{2s}}d\tau\\
&\lesssim\sup\limits_{\tau_1}
\sum\limits_{n}\int\limits_{-\infty}\limits^{+\infty} \frac{\lan
n\ran^{2r-2s-2}}{ \lan\tau+n^2\ran \lan\tau_2\ran^{1-2\theta}}d\tau\\
&\lesssim\sup\limits_{\tau_1} \sum\limits_{n}
\frac{1}{\lan\tau_1+n^2\ran^{1-2\theta}}\lesssim 1,
\end{split}
\end{equation*}
since $r \geq 0$,  $r-s \le 1$ and $0<\theta <1/4$.

Finally, we bound~(\ref{e.w3}) by noting that, in the region $A_2$
it holds $|n|>100$,  $|n_1|<|n|/2$ and $\cL=|\tau_2|$. Then,
$|n_2|\sim |n|$ and the dispersive
relation~(\ref{e.dispersive-uv-periodic}) yield that
$|\tau_2|\gtrsim n^2$. Using these conditions and that $r\geq 0$,
$r-s \le 1$ we obtain
\begin{equation*}
\begin{split}
\widetilde{W}_2&=\sup\limits_{n_2,\tau_2} \frac{1}{\lan
n_2\ran^{2s} \lan\tau_2\ran} \sum\limits_{n}\int\limits_{A_2}
\frac{\lan n\ran^{2r}}{\lan\tau+n^2\ran
\lan\tau_1+n_1^2\ran^{1-2\theta} \lan n_1\ran^{2r}}d\tau\\
&\lesssim \sup\limits_{n_2,\tau_2}
\sum\limits_{n}\int\limits_{-\infty}\limits^{+\infty}\frac{\lan n
\ran^{2r-2s-2}}{\lan\tau+n^2\ran
\lan\tau_1+n_1^2\ran^{1-2\theta}}d\tau\\
&\lesssim \sup\limits_{n_2,\tau_2} \sum \limits_{n}\frac{1}{\lan
2n_2(n+\frac{\tau_2}{2n_2}-\frac{n_2}{2})\ran^{1-2\theta}}\\
&=\sup\limits_{n_2,\tau_2}\left \{\sum \limits_{n\in
H_1}\frac{1}{\lan
2n_2(n+\frac{\tau_2}{2n_2}-\frac{n_2}{2})\ran^{1-2\theta}} + \sum
\limits_{n\in H_2}\frac{1}{\lan
2n_2(n+\frac{\tau_2}{2n_2}-\frac{n_2}{2})\ran^{1-2\theta}}\right
\},
\end{split}
\end{equation*}
where {\small $$H_1:=\left \{n\in
\Z:\;\left|n+\frac{\tau_2}{2n_2}-\frac{n_2}{2}\right|< 2\right
\}\;\;\text{and}\;\; H_2:=\left\{n\in
\Z:\;\left|n+\frac{\tau_2}{2n_2}-\frac{n_2}{2}\right|\ge 2
\right\}.$$} Now we note that $\# H_1\le 4$ and for any $n\in H_2$
we have
$$\lan 2n_2(n+\tfrac{\tau_2}{2n_2}-\tfrac{n_2}{2})\ran \gtrsim
\lan n \ran \lan n+\tfrac{\tau_2}{2n_2}-\tfrac{n_2}{2}\ran,$$
since $|n_2|\sim|n|.$ Then, by  H\"older's inequality
\begin{equation*}\begin{split}
&\sum \limits_{n\in H_1}\frac{1}{\lan
2n_2(n+\frac{\tau_2}{2n_2}-\frac{n_2}{2})\ran^{1-2\theta}} + \sum
\limits_{n\in H_2}\frac{1}{\lan 2n_2(n+\frac{\tau_2}{2n_2}-\frac{n_2}{2})\ran^{1-2\theta}}\\
&\quad \le 4 + \sum \limits_{n\in H_2}\frac{1}{\lan n
\ran^{1-2\theta} \lan
n+\tfrac{\tau_2}{2n_2}-\tfrac{n_2}{2}\ran^{1-2\theta}}\\
&\quad 4+  \left(\sum \limits_{n}\frac{1}{\lan n
\ran^{2(1-2\theta)}}\right)^{1/2}\left(\sum
\limits_{n}\frac{1}{\lan
n+\tfrac{\tau_2}{2n_2}-\tfrac{n_2}{2}\ran^{2(1-2\theta)}}\right)^{1/2}\lesssim
1,
\end{split}\end{equation*}
since  and $0<\theta <1/4$. This completes the proof of
(\ref{Estimativa-uv}).

Next, we prove (\ref{Estimativa-uv-complement}). We let $a\in
(1/2,\;3/4-\theta)$. By using Cauchy-Schwarz inequality, we have
that
\begin{equation}\begin{split}\label{Estimativa-uv-complement-1}
\left \|\lan n\ran^{r}\frac{\widehat{uv}(n,\tau)}{\lan \tau
+n^2\ran } \right \|_{\ell^2_nL^1_{\tau}}^2 &\le  \sum \limits_
{n}\lan n\ran^{2r}
\Bigl\{\int\limits_{-\infty}^{+\infty}\frac{|\widehat{uv}(n,\tau)|^2}{\lan
\tau + n^2\ran^{2(1-a)}}d\tau
\int\limits_{-\infty}^{+\infty}\frac{d\tau}{\lan \tau +
n^2\ran^{2a}}\Bigl \}.
\end{split}\end{equation}
Now, we separate $\Z^2\times\R^2$ in the same regions used to
estimate (\ref{Estimativa-uv}) and we note that, except in the
region $A_{0,3}$, the right-hand of
(\ref{Estimativa-uv-complement-1}) can be estimated in the same
way that (\ref{Estimativa-uv}). To see this, we observe that the
integral $\int_{-\infty}^{+\infty}\frac{d\tau}{\lan \tau +
n^2\ran^{2a}}$ is convergent and we replace the term $\lan \tau +
n^2 \ran$ by $\lan \tau + n^2\ran^{2(1-a)}$ in (\ref{e.w1}),
(\ref{e.w2}) and (\ref{e.w3}), then we follows the same steps to
bound the corresponding expressions in each region, using that the
condition $2(1-a)+(1-2\theta)-1>1/2$ holds for $a\in
(1/2,\;3/4-\theta)$.

Now we proceed with the estimate of the right-hand of
(\ref{Estimativa-uv-complement-1}) in  $A_{0,3}$. Here, by using
the fact that $|\tau + n^2|\gtrsim |n|^2$ we have that
\begin{equation}\label{Estimativa-uv-complement-2}
\int\limits_{A_{0,3}}\frac{d\tau}{\lan \tau +
n^2\ran^{2a}}\lesssim \lan n \ran ^{2(1-2a)}.
\end{equation}
Then, using  (\ref{Estimativa-uv-complement-2}), we have
\begin{equation}\label{Estimativa-uv-complement-3}
\left \|\lan n\ran^{r}\frac{\widehat{uv}(n,\tau)}{\lan \tau
+n^2\ran } \chi_{A_{0,3}}\right \|_{\ell^2_nL^1_{\tau}}^2\lesssim
\widetilde{W}_{0,3}\|u\|_{X_{per}^{r,1/2}}^2\|v\|^2_{H^{1/2-\theta}_t
H_{per}^{s}},
\end{equation}
where
\begin{equation}
\widetilde{W}_{0,3}=\sup\limits_{n,\tau} \frac{\lan
n\ran^{2r}n^{2(1-2a)}}{\lan\tau+n^2\ran^{2(1-a)}}
\sum\limits_{n_1}\int\limits_{A_{0,3}}
\frac{d\tau_1}{\lan\tau_1+n_1^2\ran \lan\tau_2\ran^{1-2\theta}
\lan n_1\ran^{2r} \lan n_2\ran^{2s}}.
\end{equation}
Similarly to the estimate make in (\ref{e.w1-A03}) we obtain
\begin{equation}\label{Estimativa-uv-complement-4}
\begin{split}
\widetilde{W}_{0,3}&\lesssim \sup\limits_{n,\tau} \frac{\lan
n\ran^{2r-2s+2-4a}}{\lan\tau+n^2\ran^{2(1-a)}}
\sum\limits_{n_1}\int\limits_{-\infty}\limits^{+\infty}
\frac{d\tau_1}{\lan\tau_1+n_1^2\ran
\lan\tau_2\ran^{1-2\theta} }\\
&\lesssim \sup\limits_{n,\tau} \frac{\lan n
\ran^{2r-2s+2-4a}}{\lan n\ran^{4(1-a)}}\sum\limits_{n_1}
\frac{1}{\lan\tau+n_1^2\ran^{1-2\theta}} \lesssim 1,
\end{split}
\end{equation}
since $0<\theta<1/4$ and $r-s\le 1$. Finally, combining
(\ref{Estimativa-uv-complement-2}) and
(\ref{Estimativa-uv-complement-4}) we get
\begin{equation*}\label{Estimativa-uv-complement-1}
\left \|\lan n\ran^{r}\frac{\widehat{uv}(n,\tau)}{\lan \tau
+n^2\ran } \chi_{A_{0,3}}\right \|_{\ell^2_nL^1_{\tau}} \lesssim
\|u\|_{X_{per}^{r,1/2}}\|v\|_{H^{1/2-\theta}_t H_{per}^{s}},
\end{equation*}
as we desired. Then, we finished the proof of Lemma
\ref{proposition-uv-periodic}.
\end{proof}

The next result shows that the conditions obtained above for
indices $r$ and $s$ are necessary.
\begin{proposition} For any real numbers $b_1$ and  $b_2$, the veracity of the inequality
$$\|uv\|_{X^{r,-1/2}}\lesssim \|u\|_{X^{r,b_1}} \|v\|_{H_t^{b_2}
H_x^{s}}$$ implies that $\max\{0, r-1\}\le s$.
\end{proposition}

\begin{proof} Firstly, we fix $N\gg 1$ a large integer and define
de sequences
$$
\alpha_1(n)=
\begin{cases}
1 & \textrm{if $n=N$},\\
0 & \textrm{otherwise}
\end{cases}
\quad \text{and} \quad \beta_1(n)=
\begin{cases}
1 & \textrm{if $n=-2N$},\\
0 & \textrm{otherwise}.
\end{cases}
$$
Let $u_{1_N}(x,t)$ and $v_{1_N}(x,t)$ be given by
$\widehat{u}_{1_N}(n,\tau)=\alpha_1(n)\chi_{[-1,1]}(\tau+n^2)$\;
and\; $\widehat{v}_{1_N}(n,\tau)=\beta_1(n)\chi_{[-1,1]}(\tau)$.
Taking into account the dispersive relation
$$\tau+n^2-(\tau_1+n_1^2)-\tau_2 = n^2-n_1^2,$$ we can easily
compute that
\begin{equation*}
\|u_{1_N}v_{1_N}\|_{X^{r,-1/2}} \sim N^r,\quad
\|u_{1_N}\|_{X^{r,b_1}} \sim N^r\quad  \text{and}\quad
\|v_{1_N}\|_{H_t^{b_2}H_x^s} \sim N^s
\end{equation*}
Hence, from the bound $\|u_{1_N}v_{1_N}\|_{X^{r,-1/2}}\lesssim
\|u_{1_N}\|_{X^{r,b_1}} \|v_{1_N}\|_{H_t^{b_2} H_x^{s}}$ we must
have  $N^r\lesssim N^{r+s}$ for $N \gg 1$, which implies that
$s\geq 0$.

Secondly, we define the sequences
$$
\alpha_2(n)=
\begin{cases}
1 & \textrm{if $n=N$},\\
0 & \textrm{otherwise}
\end{cases}
\quad \text{and}\quad \beta_2(n)=
\begin{cases}
1 & \textrm{if $n=0$},\\
0 & \textrm{otherwise}.
\end{cases}
$$
Let $\widehat{u}_{2_N}(n,\tau)=\alpha_2(n)\chi_{[-1,1]}(\tau+n^2)$
and $\widehat{v}_{2_N}(n,\tau)=\beta_2(n)\chi_{[-1,1]}(\tau)$.
Again, it is easy to see that
$$
\|u_{2_N}v_{2_N}\|_{X^{r,-1/2}} \sim N^{r-1},\quad
\|u_{2_N}\|_{X^{r,b_1}} \sim 1 \quad \text{and}\quad
\|v_{2_N}\|_{H_t^{b_2}H_x^s} \sim N^s
$$
Hence, the bound $\|u_{2_N}v_{2_N}\|_{X^{r,-1/2}}\lesssim
\|u_{2_N}\|_{X^{r,b_1}} \|v_{2_N}\|_{H_t^{b_2} H_x^{s}}$ implies
$N^{r-1}\lesssim N^s$ for $N \gg 1$, so we must have $r-1\le s$.

\end{proof}

\begin{lemma}\label{p.derivate-u2-periodic} Let $0<\theta < 1/4$. Then,
the following estimates
\begin{equation}\label{Estimativa-Derivada-u1u2}
\|\p_x(u\bar{w})\|_{H_t^{-1/2} H_{per}^s}\lesssim
\|u\|_{X_{per}^{r,1/2-\theta}}\|w\|_{X_{per}^{r,1/2}}
+\|u\|_{X_{per}^{r,1/2}}\|w\|_{X_{per}^{r,1/2-\theta}}
\end{equation}
{\small
\begin{equation}\label{Estimativa-Derivada-u1u2-complement}
\left \|\lan n\ran^{s}\frac{\widehat{\p_x(u\bar{w})}(n,\tau)}{\lan
\tau \ran }\right \|_{\ell^2_nL^1_{\tau}}\lesssim
\|u\|_{X_{per}^{r,1/2-\theta}}\|w\|_{X_{per}^{r,1/2}}
+\|u\|_{X_{per}^{r,1/2}}\|w\|_{X_{per}^{r,1/2-\theta}}
\end{equation}}
hold provided $0\le s\le \min\{2r-1,\; r\}$.
\end{lemma}

\begin{proof}
The proof is similar to Lemma (\ref{proposition-uv-periodic}).
Here, the relevant dispersive relation is given by
\begin{equation}\label{e.u2-dispersion}
(\tau_1+n_1^2) + (\tau_2-n_2^2) -\tau= n_1^2-n_2^2,
\end{equation}
where $\tau _2= \tau-\tau_1$ and $n_2=n-n_1$.

To prove (\ref{Estimativa-Derivada-u1u2}), by duality arguments,
it suffices to bound uniformly the follo\-wing expressions:

\begin{align}
&\label{e.u2-2}Z_0=\sup\limits_{n_1,\tau_1} \frac{1}{\lan n_1
\ran^{2r} \lan \tau_1+n_1^2 \ran} \sum\limits_{n}\int\limits_{C_0}
\frac{|n|^2\lan n \ran^{2s}}{\lan \tau \ran \lan
\tau_2-n_2^2\ran^{1-2\theta} \lan n_2 \ran^{2r}}d\tau ,\\
&\label{e.u2-1} Z_1=\sup\limits_{n,\tau} \frac{|n|^2\lan n
\ran^{2s}}{\lan \tau \ran} \sum\limits_{n_1} \int\limits_{C_1}
\frac{d\tau_1}{\lan \tau_1+n_1^2 \ran \lan \tau_2-n_2^2
\ran^{1-2\theta} \lan n_1 \ran^{2r} \lan n_2 \ran^{2r}},\\
&\label{e.u2-3} Z_2=\sup\limits_{n_2,\tau_2} \frac{1}{\lan n_2
\ran^{2r} \lan \tau_2-n_2^2 \ran} \sum\limits_{n}\int\limits_{C_2}
\frac{|n|^2\lan n \ran^{2s}}{\lan \tau \ran \lan \tau_1+n_1^2
\ran^{1-2\theta}\lan n_1 \ran^{2r}}d\tau,
\end{align}
where $C_0$, $C_1$ and $C_2$ are defined as follows. We denote by
$$\cL:= \max \bigl \{|\tau|, |\tau_1+n_1^2|,|\tau_2-n_2^2| \bigl \}$$
and then we define de following sets:
\begin{equation*}
\begin{split}
&C_{0,1}:=\{(n,\tau,n_1,\tau_1): |n|\le 100\}, \\
&C_{0,2}:=\left \{(n,\tau,n_1,\tau_1): |n| > 100, \tfrac{|n_2|}{2}\le |n_1|\le 2|n_2|\right\}, \\
&C_{0,3}:=\left \{(n,\tau,n_1,\tau_1): |n|>100, |n_1| <
\tfrac{|n_2|}{2}\; \text{or}\; |n_2| < \tfrac{|n_1|}{2} \textrm{
and } \cL=|\tau_1+n_1^2|\right\}.
\end{split}
\end{equation*}
Now we put
\begin{equation*}
\begin{split}
&C_0:=C_{0,1} \cup C_{0,2}\cup C_{0,3},\\
&C_1:= \left \{(n,\tau,n_1,\tau_1): |n|>100,
|n_1|<\tfrac{|n_2|}{2}\; \text{or}\; |n_2| < \tfrac{|n_1|}{2}
\textrm{
and}\; \cL=|\tau|\right\},\\
&C_2:=\left\{(n,\tau,n_1,\tau_1): |n|>100,
|n_1|<\tfrac{|n_2|}{2}\; \text{or}\; |n_2| < \tfrac{|n_1|}{2}
\textrm{ and } \cL=|\tau_2-n_2^2|
 \right\}.
\end{split}
\end{equation*}

Now, we bound~(\ref{e.u2-2}). In the region $C_{0,1}$, it holds
$|n|\le 100$. Hence,
\begin{equation*}
\begin{split}
Z_{0,1}&:=\sup\limits_{n_1,\tau_1} \frac{1}{\lan n_1 \ran^{2r}
\lan \tau_1+n_1^2 \ran} \sum\limits_{n}\int\limits_{C_{0,1}}
\frac{|n|^2\lan n \ran^{2s}}{\lan \tau \ran \lan
\tau_2-n_2^2\ran^{1-2\theta} \lan n_2 \ran^{2r}}d\tau\\
&\lesssim \sup\limits_{n_1,\tau_1} \sum\limits_{|n|\le 100}\;\int
\limits_{-\infty}^{+\infty}
\frac{d\tau }{\lan \tau \ran \lan \tau_2-n_2^2\ran^{1-2\theta} }\\
&\lesssim \sup\limits_{n_1,\tau_1} \sum\limits_{|n|\le
100}\frac{1}{\lan \tau_1 + (n-n_1)^2\ran^{1-2\theta}} \lesssim 1,
\end{split}
\end{equation*}
since $r \ge 0$ and $1-2\theta>0.$

In the region $C_{0,2}$, we have that $|n_1|\sim |n_2|$. Hence,
\begin{equation*}
\begin{split}
Z_{0,2}&:=\sup\limits_{n_1,\tau_1} \frac{1}{\lan n_1 \ran^{2r}
\lan \tau_1+n_1^2 \ran} \sum\limits_{n}\int\limits_{C_{0,2}}
\frac{|n|^2\lan n \ran^{2s}}{\lan \tau
\ran \lan \tau_2-n_2^2\ran^{1-2\theta} \lan n_2 \ran^{2r}}d\tau\\
& \lesssim \sup\limits_{n_1,\tau_1} \frac{\lan n_1
\ran^{2s-4r+2}}{\lan \tau_1+n_1^2 \ran}
\sum\limits_{n}\int\limits_{-\infty}^{+\infty} \frac{d\tau}{\lan
\tau
\ran \lan \tau_2-n_2^2\ran^{1-2\theta} } \\
&\lesssim \sup\limits_{n_1,\tau_1} \sum\limits_{n}\frac{1}{\lan
\tau_1 + (n-n_1)^2\ran^{1-2\theta}} \lesssim 1,
\end{split}
\end{equation*}
for $0\le s\leq 2r-1$ and $0<\theta < 1/4$.

In the region $C_{0,3}$, the dispersion
relation~(\ref{e.u2-dispersion})and the assumptions $|n_1|\nsim
|n_2|$, $|n|\ge 100$ and $\cL=|\tau_1+n_1^2|$ imply that
$|\tau_1+n_1^2| \gtrsim (\max\{|n_1|, |n_2|\})^2$. Then,
\begin{equation*}
\begin{split}
Z_{0,3}&:=\sup\limits_{n_1,\tau_1} \frac{1}{\lan n_1 \ran^{2r}
\lan \tau_1+n_1^2 \ran} \sum\limits_{n}\int\limits_{C_{0,3}}
\frac{|n|^2\lan n \ran^{2s}}{\lan \tau
\ran \lan \tau_2-n_2^2\ran^{1-2\theta} \lan n_2 \ran^{2r}}d\tau\\
&\lesssim \sup\limits_{n_1,\tau_1}
\sum\limits_{n}\int\limits_{-\infty}^{+\infty}
\frac{\bigl\lan\max\{|n_1|, |n_2|\} \bigl\ran^{2s-2r}}{\lan \tau
\ran \lan \tau_2-n_2^2\ran^{1-2\theta}}d\tau\\
&\lesssim \sup\limits_{n_1,\tau_1} \sum\limits_{n}\frac{1}{\lan
\tau_1 + (n-n_1)^2\ran^{1-2\theta} } \lesssim 1,
\end{split}
\end{equation*}
for  $0\le s \le r$ and $0<\theta < 1/4$. Then, the inequality
$|Z_0|\le |Z_{0,1}|+|Z_{0,2}|+ |Z_{0,3}|\lesssim 1$ yields the
desired estimate for $Z_0$.

The contribution of~(\ref{e.u2-1}) can be estimated as follows. In
the region $C_1$, we have that $|n|\sim \max\{|n_1|, |n_2|\}$ and
$|\tau|\geq (\max\{|n_1|, |n_2|\})^2$. Thus,
\begin{equation*}
\begin{split}
Z_1&\le \sup\limits_{n,\tau} \frac{\lan n \ran^{2s+2}}{\lan \tau
\ran} \sum\limits_{n_1} \int\limits_{C_1}\frac{d\tau_1}{\lan
\tau_1+n_1^2 \ran
\lan \tau_2-n_2^2 \ran^{1-2\theta} \lan n_1 \ran^{2r} \lan n_2 \ran^{2r}}\\
&\lesssim \sup\limits_{n,\tau}  \sum\limits_{n_1}
\int\limits_{-\infty}^{\infty} \frac{ \bigl\lan\max\{|n_1|,
|n_2|\} \bigl\ran^{2s-2r}}{\lan \tau_1+n_1^2 \ran
\lan \tau_2-n_2^2 \ran^{1-2\theta}}d\tau_1\\
&\lesssim \sup\limits_{n,\tau} \sum\limits_{n_1} \frac{1}{\lan
\tau + n_1^2 - n_2^2 \ran^{1-2\theta}}\\
&\lesssim \sup\limits_{n,\tau} \sum\limits_{n_1} \frac{1}{\lan
2nn_1 +\tau - n^2 \ran^{1-2\theta}}\lesssim 1,
\end{split}
\end{equation*}
for $0\le s\le r$ and $0<\theta<1/4,$ using the same arguments to
estimate $\widetilde{W}_2$ in Lemma \ref{proposition-uv-periodic}.

On the other hand, the expression~(\ref{e.u2-3}) can be controlled
by using that in the region $C_2$ hold $|n|\sim \max\{|n_1|,
|n_2|\}$ and $|\tau_2-n_2^2|\gtrsim (\max\{|n_1|, |n_2|\})^2$.
Then,
\begin{equation*}
\begin{split}
Z_2&= \sup\limits_{n_2,\tau_2}\frac{1}{\lan n_2 \ran^{2r}\lan
\tau_2-n_2^2\ran} \sum\limits_{n}\int \limits_{C_2}
\frac{|n|^2\lan n \ran^{2s}}{\lan
\tau \ran \lan \tau_1+n_1^2 \ran^{1-2\theta}\lan n_1 \ran^{2r}}d\tau\\
&\lesssim
\sup\limits_{n_2,\tau_2}\sum\limits_{n}\int\limits_{-\infty}^{+\infty}
\frac{\bigl\lan\max\{|n_1|, |n_2|\} \bigl\ran^{2s-2r}}{\lan
\tau \ran \lan \tau_1+n_1^2 \ran^{1-2\theta}}d\tau\\
&\lesssim \sup\limits_{n_2,\tau_2} \sum\limits_{n} \frac{1}{\lan
(n+n_2)^2 -\tau_2 \ran^{1-2\theta}} \lesssim 1,
\end{split}
\end{equation*}
for $s\le r$ and $0<\theta<1/4$. Collecting all the estimates
above we obtain the claimed estimate
(\ref{Estimativa-Derivada-u1u2}).

The prove of (\ref{Estimativa-Derivada-u1u2-complement}) follows
from a similar way to the proof of
(\ref{Estimativa-uv-complement}).
\end{proof}

Now we exhibit examples showing the necessity of the conditions
for $r$ and $s$ used in Lemma \ref{p.derivate-u2-periodic}.
\begin{proposition} For any real numbers $b_1$ and $b_2$ the veracity of
the inequality $$\|\p_x(u \bar{w})\|_{H_t^{-1/2}
H_{per}^s}\lesssim \|u\|_{X^{r,b_1}} \|w\|_{X^{r,b_2}}$$ implies
that $s\leq \min \{2r-1,\;r\}$.
\end{proposition}

\begin{proof}For a fixed large integer $N\gg 1$, we define de
following sequences:
$$
\alpha_1(n)=
\begin{cases}
1 &\textrm{if $n=N$},\\
0 & \textrm{otherwise}
\end{cases}
\quad\text{and}\quad \beta_1(n)=
\begin{cases}
1 & \textrm{if $n=-N$},\\
0 & \textrm{otherwise}.
\end{cases}
$$
Putting $\widehat{u}_{1_N}(n,\tau) = \alpha_1(n)
\chi_{[-1,1]}(\tau+n^2)$\; and\; $\widehat{w}_{1_N}(n,\tau) =
\beta_1(n) \chi_{[-1,1]}(\tau+n^2)$, a simple calculation using
the dispersive relation (\ref{e.u2-dispersion}) gives that
\begin{equation*}
\|(u_1\bar{w}_1)_x\|_{H_t^{-1/2} H_{per}^s}\sim N^{s+1} \quad
\textrm{and}\quad \|u_1\|_{X^{r,b_1}}\sim N^r \sim
\|w_1\|_{X^{r,b_2}}.
\end{equation*}
Hence, the inequality
$\|(u_1\bar{w}_1)_x\|_{H_t^{-1/2}H_{per}^s}\lesssim
\|u_1\|_{X^{r,b_1}} \|w_1\|_{X^{r,b_2}}$ implies
$$N^{s+1}\le N^{2r},\;\text{for}\; N\gg 1 \Longleftrightarrow s\le 2r-1.$$
Finally, we define
$$
\alpha_2(n)=
\begin{cases}
1 & \textrm{if $n=0$},\\
0 & \textrm{otherwise}
\end{cases}
\quad \text{and} \quad \beta_2(n)=
\begin{cases}
1 & \textrm{if $n=N$},\\
0 & \textrm{otherwise}.
\end{cases}
$$
and we put $\widehat{u}_{2_N}(n,\tau) = \alpha_2(n)
\chi_{[-1,1]}(\tau+n^2)$\; and\;  $\widehat{w}_{2_N}(n,\tau) =
\beta_2(n) \chi_{[-1,1]}(\tau+n^2)$. Then, by similar calculations
as in the previous case we obtain
\begin{equation*}
\|(u_2\bar{w}_2)_x\|_{H_t^{-1/2}H_{per}^s}\sim N^{s}, \quad
\|u_2\|_{X^{r,b_1}}\sim  1\quad \textrm{and} \sim
\|w_2\|_{X^{r,b_2}}\sim  N^r.
\end{equation*}
Again, the inequality
$\|(u_2\bar{w}_2)_x\|_{H_t^{-1/2}H_{per}^s}\lesssim
\|u_2\|_{X^{r,b_1}} \|w_2\|_{X^{r,b_2}}$ implies
$$N^{s}\le N^{r},\;\text{for}\; N\gg 1 \Longleftrightarrow s \le r.$$
Thus, we finished the proof.
\end{proof}
\subsection{Proof of Local Theorem} The next lemmas will be useful
in the proof of  Theorem \ref{local-theorem-periodic}.
\begin{lemma}\label{NLE-Borgain-Spaces} For any $s\in \R$,\;$\delta \in (0,1]$, $0< \mu <1/2$ and
$-1/2<b_1\le b_2 <1/2$ we have
\begin{enumerate}
\item [(a)] {\small $\|\eta_{\delta}(\cdot)F\|_{X^{s,1/2}_{per}} \le C
\delta^{-\mu}\|F\|_{X^{s,1/2}_{per}}$}\; and\; {\small
$\|\eta_{\delta}(\cdot)F\|_{H^{1/2}_tH^s_{per}} \le C\delta^{-\mu}
\|F\|_{H^{1/2}_tH^s_{per}}$;}

\vspace{0.3cm}
\item [(b)] {\small $\|\eta_{\delta}(\cdot)F\|_{X^{s,b_1}_{per}} \le C
\delta^{b_2-b_1}\|F\|_{X^{s,b_2}_{per}}$}\; and\; {\small
$\|\eta_{\delta}(\cdot)F\|_{H^{b_1}_tH^s_{per}} \le
C\delta^{b_2-b_1} \|F\|_{H^{b_2}_tH^s_{per}}$.}
\end{enumerate}
\end{lemma}
\begin{proof}
The proof of this result can be found, for instance, in
\cite{Takaoka} and \cite{ACM}.
\end{proof}
\begin{lemma}[Trilinear Estimate]\label{Trilinesr-Estimate} For any $s\ge 0$, we have
$$\|uv\bar{w}\|_{X^r_{per}}\lesssim \|u\|_{X^{s,3/8}_{per}}\|v\|_{X^{s,3/8}_{per}}\|w\|_{X^{s,3/8}_{per}}$$
\end{lemma}
\begin{proof}
See \cite{Bourgain} and \cite{ACM}.
\end{proof}

Now we give the sketch of the proof of local theorem. First, we
let $(u_0, v_0)\in H^r_{per}\times H^s_{per}$ where $r$ and $s$
satisfying
$$\max\{0,\,r-1\}\le s \le \min\{r,\,2r-1\}$$ and we consider the
operator $\Phi=(\Phi_1,\,\Phi_2)$, with
\begin{equation}\begin{split}
&\Phi_1(u,v)=\eta(t)u_0-i\eta(t)\int_0^te^{i(t-t')\partial_x^2}
\left((\eta_{\delta}u\eta_{\delta}v)(t') + \eta_{\delta}u|\eta_\delta u|^2(t') \right)dt',\\
&\Phi_2(u,v)=\eta(t)v_0+
\eta(t)\int_0^t\p_x(|\eta_{\delta}u|^2)(t')dt',
\end{split}\end{equation}
defined on the ball
$$\cB[a,b]=\left \{(u,v)\in  X_{per}^{r}
\times Y_{per}^{s}:\;\|u\|_{X_{per}^{r}}\le a \text{ and }
\|v\|_{Y_{per}^s}\le b\right \}.$$

Then, by Lemmas \ref{proposition-uv-periodic},
\ref{p.derivate-u2-periodic}, \ref{NLE-Borgain-Spaces} and
\ref{Trilinesr-Estimate} we have
\begin{equation}\begin{split}
\|\Phi_1(u,v)\|_{X_{per}^{r}}& \le C_0\|u_0\|_{H^r_{per}}
+C\Bigl (\|\eta_{\delta}u\|_{X_{per}^{r,1/2-\theta}}\|\eta_{\delta}v\|_{H_t^{1/2}H^s_{per}}+\\
&\quad\quad\quad
+\|\eta_{\delta}u\|_{X_{per}^{r,1/2}}\|\eta_{\delta}v\|_{H_t^{1/2-\theta}H^s_{per}}
+ \|\eta_{\delta}u\|^3_{X_{per}^{r,3/8}}\Bigl)\\
&\le C_0\|u_0\|_{H^r_{per}}+C\delta^{\epsilon}(a b + a^3)
\end{split}\end{equation} and
\begin{equation}\begin{split}
\|\Phi_2(u,v)\|_{Y^s_{per}}&\le C_0\|v_0\|_{H^s_{per}}+
C\left(\|\eta_{\delta}u\|_{X_{per}^{r,1/2-\theta}}\|\eta_{\delta}u\|_{X_{per}^{r,1/2}}
\right)\\
&\le C_0\|v_0\|_{H^s_{per}}+C\delta^{\epsilon}a^2,
\end{split}\end{equation}
with $\epsilon
$ enough small.

Now we  put $a=2C_0\|u_0\|_{H_{per}^r}$ and
$b=2C_0\|v_0\|_{H^s_{per}}$ and then we let $\delta$ such that
$\delta^{\epsilon}\le \min
\left\{\frac{1}{2C(ab+a^3)},\;\frac{1}{2Ca^2}\right\}$. Thus, we
have that $\Phi(\cB[a,b])\subset \cB[a,b].$ The contraction
condition
$$\|\Phi(u,v)-\Phi(\tilde u,\tilde v)\|^{r\times s}_{per}
\le C(a,b)\delta^{\theta}\|(u-\tilde u,v-\tilde v)\|^{r\times
s}_{per},$$ where $\|(f,g)\|^{r\times
s}_{per}:=\|f\|_{X^r_{per}}+\|g\|_{Y^s_{per}}$ and $C(a,b)$ is a
positive constant depending only on $a$ and $b$, follows
similarly. This shows that the map $\Phi$ is a contraction on
$\cB[a,b]$. There we obtain a unique fixed point which solves the
system for $T< \delta$ and we finish the proof.
\begin{remark}
We note that global well-posedness in $H^1_{per}\times L^2_{per}$
follows directly of the local theorem for $(r,s)=(1,0)$ combined
with the conservation laws (\ref{CL-1}), (\ref{CL-2}) and
(\ref{CL-3}).
\end{remark}

\section{Ill-posedness}

In this section we will show that the solution of (\ref{benney})
cannot depend uniformly continuously on its initial data for $r<0$
and $s\in \R$. We will use the same argument given in \cite{BGT}.

\subsection{Proof of theorem \ref{ill-posedness-theorem}}
 It is easy to check that
   \begin{equation}\label{52}
     \begin{array}{ll}
     u_{N,a}(t,x)=a\exp(iNx) \exp(-it(N^{2}+(\gamma+\beta)a^2)) \\
     \\
     v_{N,a}(t,x)=\gamma a^2 ,
     \end{array}
   \end{equation}
 where $a\in \R$ and $N$ is any positive integer, solves (\ref{benney}) with initial data $u_0(x)=a\exp(iNx) $ and
 $v_0(x)=\gamma a^2$.
 Moreover, for $a=\alpha (1+N^2)^{\frac{r}{2}}$, where $\alpha $ is a real constant, and $|\gamma |= (1+N^2)^{r}$ we have
 $$||u_0(x)||_{H^r}^2 \leq c\alpha^2,\;\;\text{and}\;\;||v_0(x)||_{H^s}^2 \leq c\alpha^4$$
where $c$ is a constant. Let $a_1=\alpha_1(1+N^{2})^{\frac{r}{2}}$ and $ a_2=\alpha_2(1+N^{2})^{\frac{r}{2}}$. For the Sobolev norm of the difference of two initial data, we have

   $$||u_{N,a_1}(0)-u_{N,a_2}(0)||_{H^r}^{2}=c|\alpha_1-\alpha_2|^{2} \rightarrow 0, \rm{as} \; \alpha_1 \rightarrow \alpha_2$$
  and $$||v_{a_1}(0)-v_{a_2}(0)||_{H^s}^2=|\gamma|^{2}|\alpha_1^2-\alpha_2^2|^{2}(1+N^{2})^{-2r}=|\alpha_1^2-\alpha_2^2|^2,\; \rm{as} \; \alpha_1 \rightarrow \alpha_2.$$
   On the other hand we have
     $$  \begin{array}{ll}
     ||u_{N,a_1}(t,x)-&u_{N,a_2}(t,x)||_{H^r}^2  = \displaystyle \sum_{n=-\infty}^{+\infty}{(1+|n |^{2})^{r}|\hat{u}_{N, \alpha_1 }(n)-\hat{u}_{N, \alpha_2}(n)|^{2}} \\
     \\
     &=(1+N^{2})^{r}|a_1e^{-it(N^{2}+(\gamma+\beta)a_{1}^{2})}-a_2e^{-it(N^{2}+(\gamma+\beta)a_{2}^{2})}|^{2}\\
     \\
     &=|\alpha_1-\alpha_2e^{it(\gamma+\beta)(\alpha_1^2-\alpha_2^2)(1+N^2)^{-r}}|^2
     \end{array}
     $$
   Let $r<0$, and   $\alpha_1$ and $\alpha_2$ are such that $\beta (\alpha_1^2-\alpha_2^2)(1+N^2)^{-r}=\delta (1+N^2)^{\nu},$
   where $\nu >0, $ and $\nu+r<0$. Then for  $t={\frac{\pi}{2}}(\delta^{-1}(1+N^2)^{-\nu})$ we have
   $$||u_{N,a_1}(t,x)-u_{N,a_2}(t,x)||_{H^r}^2 \geq c(\alpha_1^2+\alpha_2^2).$$
   Note that $t$ can made arbitrary small, by choosing $N$ sufficiently large.

\section{Existence of periodic travelling wave solutions}

We are interesting in this section in finding explicit solutions
for (\ref{benney}) in the form
     \begin{equation}\label{21}
      \begin{cases}
         u(t,x)=e^{-i\omega
         t}e^{i{\frac{c}{2}}(x-ct)}\varphi_{\omega, c}(x-ct),\\
         v(t,x)=n_{\omega,c}(x-ct),
       \end{cases}
     \end{equation}
   where
   $\varphi_{\omega,c}$\; and\; $ n_{\omega,c}$ are smooth and
   $L$-periodic functions, $c>0, \; \; \omega \in \mathbb{R}$ and suppose that there is a $q\in \mathbb{N}$ such that
   ${\frac{4\pi q}{c}}=L$. So, putting (\ref{21}) into (\ref{benney}) we obtain
    \begin{equation}\label{23}
      \begin{cases}
           \varphi_{\omega,c}^{''}+
           (\omega+{\frac{c^2}{4}})\varphi_{\omega,c}=
           \varphi_{\omega,c}n_{\omega,c}+\beta
           \varphi_{\omega,c}^3,\\
           -cn_{\omega,c}^{'}=2\varphi_{\omega,c}\varphi_{\omega,c}^{'}.
        \end{cases}
   \end{equation}
 If $n_{\omega,c}=\gamma \varphi_{\omega,c}^2$, then from the
 second equation in (\ref{23}) we have $\gamma={-\frac{1}{c}}$.
 Substituting $n_{\omega,c}$ in the first equation in (\ref{23}),
 it follows that $\varphi_{\omega,c}$ satisfies
   \begin{equation}\label{24}
     \varphi_{\omega,c}^{''}+\left(
     \omega+{\frac{c^2}{4}}\right)\varphi_{\omega,c}=\left(
     \beta-{\frac{1}{c}} \right) \varphi_{\omega,c}^3.
   \end{equation}
 If $1-\beta c>0$ and $\varphi_{\omega,c}=\left(
 {\frac{c}{1-\beta c}}\right)^{\frac{1}{2}}\phi_{\omega,c}$, then
 $\phi_{\omega,c}$ satisfies the equation
   \begin{equation}\label{25}
     \phi_{\omega,c}^{''}-\sigma
     \phi_{\omega,c}+\phi_{\omega,c}^3=0,
   \end{equation}
 where $\sigma =-\omega-{\frac{c^2}{4}}$. So, by following Angulo's arguments
in (\cite{An1}, \cite{An}) we have from (\ref{25}) that
 $\phi_{\omega,c}$ satisfies the first-order equation
   \begin{equation}\label{26}
     [\phi_{\omega,c}']^2=\frac{1}{2} P_{\phi}(\phi),
   \end{equation}
  where $P_{\phi}(t)=-t^4+2\sigma t^2+2B_{\phi}$ and $B_{\phi}$ is
  an integration constant. Let $-\eta_1< -\eta_2< \eta_2<\eta_1$
  are the zeros of the polynomial $P_{\phi}(t)$. Then
    \begin{equation}\label{27}
      [\phi_{\omega,c}^{'}]^2=\frac{1}{2}(\eta_1^2-\phi_{\omega,c}^2)(
      \phi_{\omega,c}^2-\eta_2^2).
    \end{equation}
 The solution of (\ref{27}) is
   \begin{equation}\label{28}
     \phi_{\omega,c}=\eta_1 dn\left( {\frac{\eta_1}{\sqrt{2}}}\xi ;
     \kappa \right),
   \end{equation}
 where
   \begin{equation}\label{29}
         \eta_1^2+\eta_2^2=2\sigma,\;\;
         \kappa^2={\frac{\eta_1^2-\eta_2^2}{\eta_1^2}}, \;\;
         0<\eta_2<\eta_1.
   \end{equation}
 Define the function in variable $\kappa , \; \; 0<\kappa <1,$
   $$
   K=K(\kappa)=\int_{0}^{1}{{\frac{dt}{\sqrt{(1-t^2)(1-\kappa^2t^2)}}}}
   $$
   called the complete elliptic integral of the first kind. Since
   $dn$ has fundamental period $2K(\kappa)$, it follows that
   $\phi_{\omega,c}$ has fundamental period
     $$ T_{\phi_{\omega,c}}={\frac{2\sqrt{2}}{\eta_1}}K(\kappa) $$

     Analogously as in \cite{An1} we obtain the following result.
      \begin{theorem}\label{t21}
        Let $L$  be fixed but arbitrary
        positive constant and $1-\beta c>0$, and
        $-\omega-\frac{c^2}{4}>0$.
         Let $\sigma_0>{\frac{2\pi^2}{L^2}}$
        and $\eta_{2,0}=\eta_2(\sigma_0)\in
        (0,\sqrt{\frac{\sigma_0}{3}})$ is the unique such that
        $T_{\phi}=L$. Then hold the following assertions:

        (1) There exists an interval $I(\sigma_0)$ around of
        $\sigma_0$, an interval $B(\eta_{2,0})$ around
        $\eta_{2,0}$, and a unique smooth function $\Lambda :
        I(\sigma_0) \rightarrow B(\eta_{2,0})$, such that
        $$ \Lambda(\sigma_0)=\eta_{2,0} \; \; \rm{and} \; \;
        {\frac{2\sqrt{2}}{\sqrt{2\sigma-\eta_2^2}}}K(\kappa)=L,$$
        where $\sigma \in I(\sigma_0), \eta_2=\Lambda(\sigma)$.

        (2) Solutions $(\varphi_{\omega,c}, n_{\omega,c})$ of (\ref{23}) given
        by
          \begin{equation}\label{psolu1}
            \begin{cases}
                \varphi_{\omega,c}=\sqrt{\frac{c}{1-\beta c}}\eta_1
                dn\left( {\frac{\eta_1}{\sqrt{2}}}\xi ; \kappa
                \right),\\
                n_{\omega,c}=-{\frac{\eta_1^2}{1-\beta c}}
                dn^2\left( {\frac{\eta_1}{\sqrt{2}}}\xi ; \kappa
                \right),
             \end{cases}
          \end{equation}
      with $\eta_1=\eta_1(\sigma), \; \eta_2=\eta_2(\sigma), \; \;
      \eta_1^2+\eta_2^2=2\sigma $, have the fundamental period $L$
      and satisfies (\ref{23}). Moreover, the mapping
        $$ \sigma \in I(\sigma_0) \rightarrow (\varphi_{\omega,c},
        n_{\omega,c}) $$
      is a smooth function.

       (3) $I(\sigma_0)$ can be chosen as $({\frac{2\pi^2}{L^2}},
       +\infty)$.

       (4) The mapping $\sigma \rightarrow \kappa(\sigma)$ is a
       strictly increasing function.
     \end{theorem}

  \section{Stability of travelling waves}

     In this section we consider the stability of the orbit
$$
\Omega_{(\Phi,\Psi)}=\{(e^{i\theta}\Phi(\cdot+x_0),\Psi(\cdot+x_0));\;
(\theta,x_0)\in [0,2\pi)\times \Bbb R\},
$$
in $H^1_{per}([0,L])\times L^2_{per}([0,L])$ by the periodic flow
generated by (\ref{benney}), where we have that
$\Phi(\xi)=e^{ic\xi/2}\varphi_{\omega,c}(\xi)$,
$\Psi(\xi)=n_{\omega,c}(\xi)$, with $\varphi_{\omega,c},
n_{\omega,c}$ given in (\ref{psolu1}). Let $X$ be the space
     $X=H^1_{complex}([0,L])
     \times L^2_{real}([0,L])$, with real inner product

     $$(\vec{u_1} , \vec{u_2} )=\Re \int_{0}^{L}{(\varepsilon_1 \overline{\eta}_1 +
       \varepsilon_{1x}\overline{\eta}_{1x}+\varepsilon_2\overline{\eta_2})}dx.$$

Let $T_1 , \; \; T_2$ be one-parameter groups of unitary
  operators on $X$ defined by
    $$ \begin{array}{ll}
          T_1(s)\vec{u}(\cdot )=\vec{u}(\cdot +s) \\
          T_2(r)\vec{u}(\cdot )=(e^{-ir}\varepsilon(\cdot ),
          n(\cdot ))
       \end{array}
    $$
    for $\vec{u} \in X, \; \; s,\; r \in \mathbb{R}$. Obviously
      $$ T_1^{'}(0)=\begin{pmatrix}
          -\partial_x & 0 \\
          0 &-\partial_x
         \end{pmatrix} ,\; \; \; \;
         T_2^{'}(0)=\begin{pmatrix}
          -i & 0 \\
          0  & 0\\
         \end{pmatrix}.
      $$
      Note that the equation (\ref{benney}) is invariant under $T_1$ and
  $T_2$. If
    $$ \Phi_{\omega , c}(x)=(\varepsilon_{\omega , c}(x),
     n_{\omega , c}(x))
    $$
  where $\varepsilon_{\omega , c}(x)=e^{i{\frac
  {c}{2}}x}\varphi_{\omega, c}(x)$, then from Theorem \ref{t21}
  we obtain that
    $$ T_1(ct)T_2(\omega t)\Phi_{\omega , c}(x) $$
  is a travelling wave solution of (\ref{23}) with $ \varphi_{\omega,
  c}(x), n_{\omega , c}(x) $ defined by (\ref{psolu1}).

Now, it is easy to verify that $E_2(\vec{u})$ is invariant under
     $T_1$ and $T_2$
       \begin{equation}\label{31}
         E(T_1(s)T_2(r)\vec{u})=E(\vec{u}).
       \end{equation}
     We also have
       \begin{equation}\label{32}
         E(\vec{u}(t))=E(\vec{u}_0),
       \end{equation}
 and  that equation (\ref{benney}) can be written as the
    following Hamiltonian system
       \begin{equation}\label{33}
         {\frac {d\vec{u}}{dt}}=JE'(\vec{u})
       \end{equation}
    where $\vec{u} = (u,v)$ and $J$ is a skew-symmetric linear operator  defined by
      $$ J=\begin{pmatrix}
           -i & 0 \\
           0 & 2\partial_x
         \end{pmatrix}
      $$
    and
      $$ E'(u,v)=\begin{pmatrix}
                         -u_{xx}+u v+\beta |u|^2u \\
                         {\frac{1}{2}}|u|^2
                      \end{pmatrix}
      $$
    is the Frechet derivative of $E$. Define $B_1$ and $B_2$ such that
     $T_1'(0)=JB_1, \; \; T_2'(0)=JB_2$, then
      $$
        Q_1(\vec{u})={\frac {1}{2}}\langle B_1\vec{u}, \vec{u}
        \rangle = -{\frac{1}{4}}\int_{0}^{L}{v^2}dx+{\frac{1}{2}}Im
        \int_{0}^{L}{u_x \overline{u}}dx
      $$
      $$
        Q_2(\vec{u})={\frac {1}{2}}\langle B_2 \vec{u}, \vec{u}
        \rangle ={\frac{1}{2}}\int_{0}^{L}{|u|^2}dx.
      $$
      It is easy to verify that
       \begin{equation}\label{34}
         Q_1(T_1(s)T_2(r)\vec{u})=Q_1(\vec{u}), \; \;
         Q_2(T_1(s)T_2(r)\vec{u})=Q_2(\vec{u})
       \end{equation}
       \begin{equation}\label{35}
          Q_1(\vec{u}(t))=Q_1(\vec{u}(0)), \; \;
         Q_2(\vec{u}(t))=Q_2(\vec{u}(0))
       \end{equation}
    and
      $$
        Q_1^{'}(u,v)=\begin{pmatrix}
                          -iu_x \\
                          -\frac{1}{2}v
                       \end{pmatrix}, \; \; \;
        Q_2{'}(u,v)=\begin{pmatrix}
                          u \\
                          0 \\
                       \end{pmatrix}.
     $$
From (\ref{23}) we have
       \begin{equation}\label{36}
         E'(\Phi_{\omega , c})-cQ_1'(\Phi_{\omega ,
         c})-\omega Q_2'(\Phi_{\omega ,c})=0 .
       \end{equation}

    Define an operator from $X$ to $X^*$
      \begin{equation}\label{37}
        H_{\omega , c}=E''(\Phi_{\omega , c})-cQ_1''(\Phi_{\omega ,
         c})-\omega Q_2''(\Phi_{\omega ,c})
      \end{equation}
    and the function $d(\omega , c): \mathbb{R}\times \mathbb{R}\rightarrow
    \mathbb{R}$ by
      \begin{equation}\label{38}
        d(\omega , c)=E(\Phi_{\omega , c})-cQ_1(\Phi_{\omega ,
         c})-\omega Q_2(\Phi_{\omega ,c}).
      \end{equation}
   The operator $H_{\omega , c}$ is self-adjoint. The spectrum of
   $H_{\omega ,c}$ consists of the real numbers $\lambda $ such
   that $H_{\omega ,c}-\lambda I$ is not invertible. From  (\ref{23}) we have
     \begin{equation}\label{310}
       T_1^{'}(0)\Phi_{\omega ,c} \in Ker H_{\omega ,c} , \; \;
       T_2^{'}(0)\Phi_{\omega ,c} \in Ker H_{\omega ,c}.
     \end{equation}

   Let $Z=\{k_1T_1^{'}(0)\Phi_{\omega ,c}+k_2T_2^{'}(0)\Phi_{\omega
   ,c} \; , \; \; k_1, k_2 \in \mathbb{R}\} $. By (\ref{310}), $Z$
   is  in the kernel of $H_{\omega ,c}$.

    \textbf{Assumption} \textit{(Spectral decomposition of $H_{\omega ,c}$)} : The space $X$
    is decomposed as a direct sum
    $$
    X=N\oplus Z\oplus P
    $$
    where $Z$ is defined above, $N$ is a finite-dimensional subspace such that
      $$ \langle H_{\omega ,c} \vec{u}, \vec{u} \rangle <0  \; \;
      \rm{for}  \; \; \vec{u}\in N
      $$
    and $P$ is a closed subspace such that $ \langle H_{\omega ,c} \vec{u}, \vec{u} \rangle \geq \delta ||\vec{u}||_{X}^2$,
    for $\vec{u} \in P$ with some constant $\delta >0$ independent
    of $\vec{u}$.

    Our stability results is based in the following general theorem in \cite{GrShSt2},

    \begin{theorem}\label{t31}( Abstract Stability Theorem) Assume
    that there exists three functionals $E, Q_1, Q_2$ satisfying (\ref{31})-(\ref{35}).
    Let $n(H_{\omega ,c})$ be the number of negative eigenvalues of
    $H_{\omega ,c}$. Assume $d(\omega ,c)$ is non-degenerated at
    $(\omega, c)$ and let $p(d'')$ be the number of positive
    eigenvalues of $d''$. If $p(d'')=n(H_{\omega ,c})$, then
    the periodic travelling  wave $\Phi_{\omega ,c}(x)$ is
    orbitally stable.
    \end{theorem}

    The idea of the proof of Theorem \ref{stability theorem} is to
    apply the general Theorem \ref{t31}. Initially we identify the
    quadratic form associated to $H_{\omega,c}$. Let $\vec{z}=(e^{i\frac{c}{2}x}z_1, z_2)$, with
    $z_1=y_1+iy_2$, $y_1=Rez_1, y_2=Imz_1$. By direct computation,
    we get
     $$ \langle H_{\omega,c}\vec{z_1}, \vec{z_1}\rangle=\langle
     L_1y_1, y_1\rangle +\langle L_2y_2, y_2
     \rangle+{\frac{c}{2}}\int_{0}^{L}{\left(z_2+\frac{2}{c}\varphi_{\omega,c}
     y_1\right)^2}dx $$
    where
    $$L_1=-\partial_x^2-\left(\frac{c^2}{4}+\omega\right)+
    3\left(\beta-\frac{1}{c}\right)\varphi_{\omega,c}^2
    $$
    $$L_2=-\partial_x^2-\left(\frac{c^2}{4}+\omega\right)+
    \left(\beta-\frac{1}{c}\right)\varphi_{\omega,c}^2.
    $$
   From (\ref{23}) we also have $L_1(\partial_x\varphi_{\omega,c})=0$ and
   $L_2\varphi_{\omega,c}=0$. Consider the following periodic eigenvalue problems for $i=1,2$,
     \begin{equation}\label{311}
         \left\{
          \begin{array}{ll}
          L_if=\lambda f \\
          f(0)=f(L), \; \; f'(0)=f'(L),
         \end{array} \right.
    \end{equation}
    The problem (\ref{311}) determines a countable infinite set of eigenvalues $\{ \lambda_{n} \}$
 with $\lambda_n \rightarrow \infty$, so from  the Oscillation Theorem \cite{MaWi} we have that they are distributed  in the specific form
 $\lambda_0<\lambda_1 \leq \lambda_2< \lambda_3 \leq \lambda_4, ...$.

 For the eigenvalue problems (\ref{311}) we have the same results.

      \begin{theorem}\label{t32}
        Let $\sigma \in [ {\frac{2\pi^2}{L^2}}, +\infty )$ and
        $(\varphi_{\omega,c}, n_{\omega,c})$ be the travelling
        wave solutions of (\ref{psolu1}). Then the first three
        eigenvalues of operator $L_1$ are simple, $0$ is the
        second eigenvalue of $L_1$ with eigenfunction
        $\partial_x \varphi_{\omega,c}$. The first eigenvalue of
        the operator $L_2$ is $0$, which is simple.
     \end{theorem}

\begin{proof} Since $L_2 \varphi_{\omega, c}=0$ and $\varphi_{\omega, c}$ has no zeros on $[0,L]$, then zero is the first eigenvalue of $L_2$. Now since $L_1 \partial_x \varphi_{\omega, c}=0$ and $\partial
\varphi_{\omega, c}$ has two zeros on $[0,L)$, then it follows
that eigenvalue zero of $L_1$ is either $\lambda_1$ or
$\lambda_2$. Let $\psi=f(\theta x)$, where
$\theta^2={\frac{2}{\eta_{1}^2}}$. From equality
$\kappa^2sn^2(x)+dn^2(x)=1$ and (\ref{311}), we obtain that $\psi$
satisfies the equation
  \begin{equation}\label{314}
    \psi^{''} +(\rho-6\kappa^2 sn^2(x))\psi=0,
  \end{equation}
  where
    \begin{equation}\label{315}
     \rho = 6-{\frac{2}{\eta_{1}^2}}(\sigma-\lambda).
    \end{equation}
From Floquet theory, it follows that $(-\infty, \rho_0)$ and $(\rho_1, \rho_2)$ are instability intervals associated to the Lame's equation. Therefore the eigenvalues
$\rho_0, \rho_1$ and $\rho_2$ of (\ref{315}) are simple and the
rest of eigenvalues $\rho_3\leq \rho_4,...$ satisfies
$\rho_3=\rho_4, \rho_5=\rho_6, ...$ The eigenvalues $\rho_0,
\rho_1, \rho_2$ and its corresponding eigenfunctions $\psi_0,
\psi_1, \psi_2$ are
  $$\begin{array}{ll}
    \rho_{0}=2(1+\kappa^{2}-\sqrt{1-\kappa^{2}+\kappa^{4}}, &
     \psi_{0}=1-(1+\kappa^{2}-\sqrt{1-\kappa^{2}+\kappa^{4}})sn^{2}(x) \\
    \\
    \rho_{1}=4+\kappa^{2}, & \psi_2=sn(x)cn(x) \\
    \\
    \rho_{2}=2(1+\kappa^{2}+\sqrt{1-\kappa^{2}+\kappa^{4}}, &  \psi_{2}=1-(1+\kappa^{2}+\sqrt{1-\kappa^{2}+\kappa^{4}})sn^{2}(x)
   \end{array}
 $$
 Since $\rho_{0}<\rho_{1}$ for every $\kappa^{2} \in (0,1)$, then from (\ref{315}) we have
   $$3 \lambda_0={\frac{\eta_1^2}{2}}(\kappa^{2}-2-2\sqrt{1-\kappa^{2}+\kappa^{4}})<0$$
 Therefore $\lambda_0$ is negative eigenvalue of $L_1$ with eigenfunction $\chi_{0}(x)=\psi_{0}({\frac{x}{\theta}})$. Similarly
   $$3\lambda_2={\frac{\eta_1^2}{2}}(\kappa^{2}-2+2\sqrt{1-\kappa^{2}+\kappa^{4}})>0$$
 and $\lambda_2$ is the positive eigenvalue of $L_1$ with eigenfunction $\chi_{2}(x)= \psi_{2}({\frac{x}{\theta}})$. Thus
   $$\lambda_1={\frac{\eta_{1}^{2}(\rho_{1}-6)+2\sigma}{6}} ={\frac{\eta_{1}^{2}}{6}(4+\kappa^{2}-6+2-\kappa^{2})}=0$$
 is the second eigenvalue of $L_1$. This complete the proof of the theorem.
\end{proof}

\begin{remark}
The main properties of the spectrum of $L_1$, namely, there is
exactly a negative eigenvalue and zero is simple, it can also be
obtained via positive properties of the Fourier transform of the
solution $\varphi_{\omega,c}$ (see Angulo\&Natali \cite{AN} and Angulo \cite{AN3}).
\end{remark}

 So, from Theorem \ref{t32} we obtain immediately the following two
     results.

   \begin{lemma}\label{l31}
     For any real function $y_1 \in H^1$ satisfying
       $ \langle y_1, \chi_0\rangle=\langle y_1,
       \partial_x\varphi_{\omega,c}\rangle =0 $
    there exists a positive constant $\delta_1>0$ such that
       $ \langle L_1y_1, y_1\rangle \geq \delta_1||y_1||_{H^1}^2.
       $
   \end{lemma}

   \begin{lemma}\label{l32}
     For any real function $y_2\in H^1$ satisfying $\langle y_2, \varphi_{\omega,c}\rangle =0 $
     there exists a positive constant $\delta_2$ such that $\langle L_2 y_2, y_2 \rangle \geq \delta_2||y_2||_{H^1}^2.$
   \end{lemma}

 \begin{proof} {\bf [Theorem \ref{stability theorem}]} Choose $y_1^{-}=\chi_0, y_2^{-}=0,
   z_2^{-}=-{\frac{2}{c}}\varphi_{\omega,c}\chi_0$ and
   $\Psi^{-}=(y_1^{-}, y_2^{-}, z_2^{-})$ then
     $$
     \langle H_{\omega,c} \Psi^{-}, \Psi^{-} \rangle
     =\lambda_{0}\langle \chi_0, \chi_0 \rangle <0.
      $$
   So $H_{\omega,c}$ has a negative eigenvalue. Note that the following vectors
   $$
   \Psi_{0,1}=(\partial_x\varphi_{\omega,c},
   0,-{\frac{2}{c}}\varphi_{\omega, c} \partial_x\varphi_{\omega,c}),\;\;
   \Psi_{0,2}=(0,\varphi_{\omega,c},0)
   $$
   are in the kernel of operator $H_{\omega,c}$. Define the following subspaces associated to $H_{\omega,c}$:
   $$
   Z=\{k_1\Psi_{0,1}+k_1\Psi_{0,2} : k_1, k_2 \in
   \mathbb{R}\}
   $$
    $$
    N=\{ k\Psi^{-}: k\in \mathbb{R}\}$$
    $$ P=\{ \vec{p}\in X : \vec{p}=(p_1,p_2, p_3), \langle p_1,
    \chi_1 \rangle=\langle p_1, \partial_x\varphi_{\omega,c}\rangle =
    \langle p_2, \varphi_{\omega,c}\rangle=0\}.
    $$
For any  $\vec{u} \in X, \vec{u}=(y_1, y_2, y_2)$ choose
   $$
   a={\frac{\langle y_1, \chi_0 \rangle }{\langle \chi_0, \chi_0 \rangle }},
 \; b_{1}={\frac{ \langle \partial_{x} \varphi_{\omega, c}, y_1 \rangle }{\langle \partial_{x} \varphi_{\omega, c}, \partial_{x} \varphi_{\omega, c} \rangle}},
 \; b_{2}={\frac{\langle \varphi_{\omega, c}, y_2 \rangle }{\langle \varphi_{\omega, c}, \varphi_{\omega, c}
 \rangle}},
 $$
 then $\vec{u}$ uniquely can be represented by $
\vec{u}=a\Psi^{-}+b_1\Psi_{0,1}+b_2\Psi_{0,2}+\vec{p}$,
 where $\vec{p}\in P$. For any $\vec{p} \in P$, by  Lemmas \ref{l31} and \ref{l32}, we
     have
       $$ \langle H_{\omega,c}\vec{p}, \vec{p}\rangle \geq
       \delta_1||p_1||_{H^1}^2+\delta_1||p_2||_{H^1}^2+{\frac{c}{2}}
       \int_{0}^{L}{\left( p_3+\frac{2}{c}\varphi p_1\right)^2}dx
       $$

Next we consider the following two cases:
\begin{enumerate}

\item[(1)] If $||p_3||_{L^2} \geq {\frac{8 ||\varphi_{\omega,
c}||_{L^{\infty}}}{c}}||p_1||_{L^2}$, then

     $${\frac{c}{2}}\int_{0}^{L}{\left( p_3+{\frac{2}{c}}\varphi_{\omega, c}p_1 \right)^2}dx \geq {\frac{c}{2}}\left[ ||p_3||_{L^2}^{2}-{\frac{4}{c}}||\varphi_{\omega, c}||_{L^{\infty}}||p_1||_{L^2}||p_3||_{L^2}\right]={\frac{c}{4}}||p_3||_{L^2}^{2}$$

\item[(2)] If $||p_3||_{L^2} \leq {\frac{8 ||\varphi_{\omega, c}||_{L^{\infty}}}{c}}||p_1||_{L^2}$, then
       $$\delta_{1} ||p_1||_{H^1}^{2}\geq {\frac{\delta_1}{2}}||p_1||_{H^1}^{2}+{\frac{\delta_1}{2}}{\frac{c}{8||\varphi_{\omega, c}||_{L^{\infty}}}}||p_3||_{L^2}^{2}$$
    Thus, for any $\vec{p}\in P$, it follows that

     $$\langle H_{\omega,c}\vec{p}, \vec{p}\rangle \geq
     \delta_3||p_3||_{L^2}^2+{\frac{\delta_1}{2}}||p_1||_{H^1}^2+
     \delta_2||p_2||_{H^1}^2, $$
   where $\delta_3=\min\{ {\frac{\delta_{1} c}{16||\varphi_{\omega, c}||_{L^{\infty}}}}, {\frac{c}{4}}\}$. Finally, we have
     $$\langle H_{\omega, c} \vec{p}, \vec{p} \rangle \geq \delta ||\vec{p}||_{X}^{2},$$
   where $\delta >0$ is independent of $\vec{p}$. This proved that Assumption above is holds, and
$n(H_{\omega,c})=1$.
\end{enumerate}

Now we shall verify that $p(d'')=1$. We have
        $$
        d_c(\omega,c)=-Q_1(\Phi_{\omega,c})={\frac{1}{4(1-\beta
        c)^2}}\int_{0}^{L}{\varphi_{\omega,c}^4}dx-{\frac{c^2}{4(1-\beta
        c)}}\int_{0}^{L}{\varphi_{\omega,c}}dx $$
        $$
        d_{\omega}(\omega,c)=-Q_2(\Phi_{\omega,c})=-{\frac{c}{2(1-\beta
        c)}}\int_{0}^{L}{\varphi_{\omega,c}^2}dx.
         $$
      From equalities
        $$ \int_{0}^{L}{\varphi_{\omega,c}^2}dx={\frac{8KE}{L}},\; \;\int_{0}^{L}{\varphi_{\omega,c}^4}dx
        ={\frac{64}{L^3}}V(\kappa)
        $$
      where $E=E(\kappa)=\int_{0}^{1}{\sqrt{\frac{1-\kappa^2t^2}{1-t^2}}}dt$ is
      the complete elliptic integral of the second kind and
      $V(\kappa)={\frac{\kappa^2-1}{3}}K^4+{\frac{2}{L}}(2-\kappa^2)K^2E$,
      we obtain
        \begin{equation}\label{316}
         \begin{array}{ll}
          d_{\omega \omega}= & {\frac{4c}{L(1-\beta
        c)}}(K^{'}(\kappa)E(\kappa)+K(\kappa)E^{'}(\kappa))\kappa^{'}(\sigma),
        \\
        \\
         d_{\omega c}= & -{\frac{4}{L(1-\beta
        c)^2}}K(\kappa)E(\kappa)+{\frac{c}{2}}d_{\omega \omega}, \;\;\;
         d_{c \omega}=  -{\frac{16}{L^3(1-\beta
        c)^2}}V^{'}(\kappa)\kappa^{'}(\sigma)+{\frac{c}{2}}d_{\omega
        \omega} \\
        \\
         d_{cc}= & {\frac{32\beta}{L^3(1-\beta
        c)^3}}V(\kappa)-{\frac{8c}{L^3(1-\beta
        c)^2}}V^{'}(\kappa)\kappa^{'}(\sigma)- {\frac{2c(2-\beta
        c)}{L(1-\beta
        c)^2}}K(\kappa)E(\kappa)+{\frac{c^2}{4}}d_{\omega \omega}.
        \end{array}
        \end{equation}
 Thus
     $$ \begin{array}{ll}
         d_{cc}d_{\omega \omega}-d_{c \omega}d_{\omega c}= &
          -{\frac{64}{L^4(1-\beta
         c)^4}}V^{'}(\kappa)\kappa^{'}(\sigma)K(\kappa)E(\kappa)+\\
         \\
         & {\frac{1}{L(1-\beta
         c)}}\left[ {\frac{32\alpha}{L^2(1-\beta
         c)^2}}V(\kappa)-2cK(\kappa)E(\kappa)\right] d_{\omega
         \omega}.
       \end{array}
     $$
      We have,
     $$
     V'(\kappa)={\frac{2K^2E}{\kappa (1-\kappa^2)}}\left[ (2-\kappa^{2})E-(1-\kappa^{2})K \right],\;\;\text{and}\;\;
     {\frac{V}{L^{2}}}={\frac{\sigma (\kappa^{2}-1)}{12(2-\kappa^2)}}K^{2}+{\frac{\sigma}{6}}KE.
     $$
     Using the above estimates, we obtain
     $$
     \begin{array}{lll}
     d_{cc}d_{\omega \omega}-d_{c\omega}d_{\omega c}&=
     {\frac{4\kappa'}{L^{2}(1-\beta c)^{2}}}\left\{ -32{\frac{K^{2}}{L^{2}}}KE^{2}
     \left[(2-\kappa^{2})E-2(1-\kappa^{2})K\right]\right\}\\
     &+\left. {\frac{c}{3}}\left[ {\frac{8\beta \sigma (\kappa^{2}-1)}{2-\kappa^{2}}}K^{2}+(16
     \beta \sigma -6c(1-\beta c)^{2})KE\right] \left[ {\frac{E^{2}}
     {\kappa (1-\kappa^{2})}}-{\frac{K^{2}}{\kappa}}\right] \right\}\\
     &={\frac{4K\kappa^{'}}{L^{2}(1-\beta c)^{2}}}
     \left\{-{\frac{8 \sigma}{2-\kappa^{2}}}\left[(2-\kappa^{2})E^{3}-2(1-\kappa^{2})KE^{2}\right]\right\}\\
     &+\left. \left[{\frac{8 \beta \sigma c (\kappa^{2}-1)}{3(2-\kappa^{2})}}K
     +{\frac{c(16 \beta \sigma-6c(1-\beta c)^{2})}{3}}E\right] \left[ {\frac{E^{2}}
     {\kappa (1-\kappa^{2})}}-{\frac{K^{2}}{\kappa}}\right] \right\}
     \end{array}
     $$

   From Theorem \ref{t21}-(4), we have that $\kappa^{'}>0$. Therefore the sign
of $det(d'')=d_{cc}d_{\omega \omega}-d_{c \omega}d_{\omega c}$
depends on the sign of
$$
     \begin{array}{ll}
     B(c, \omega, \kappa ,\beta)= & \left\{-{\frac{8 \sigma}{2-\kappa^{2}}}
     \left[(2-\kappa^{2})E^{3}-2(1-\kappa^{2})KE^{2}\right]\right.\\
     \\
     &+\left. \left[{\frac{8 \beta \sigma c (\kappa^{2}-1)}{3(2-\kappa^{2})}}K
     +{\frac{c(16 \beta \sigma-6c(1-\beta c)^{2})}{3}}E\right]
     \left[ {\frac{E^{2}}{\kappa (1-\kappa^{2})}}-{\frac{K^{2}}{\kappa}}\right]
     \right\}.
     \end{array}
$$
From the relation
      \begin{equation}\label{317}
        0<{\frac{(1-\kappa^{2})K}{(2-\kappa^{2})E}}<{\frac{1}{2}}
      \end{equation}
    we get that the first term of $B(c, \omega, \kappa, \beta)$ is negative.
     Now we consider three cases for $\beta$.

    (1) Obviously if $\beta =0$, then $det(d^{''})<0$.

    (2) For $\beta<0$, using (\ref{317}), we get
       $$\begin{array}{ll}
       {\frac{8\beta \sigma (\kappa^{2}-1)K}{3(2-\kappa^2)E}}+{\frac{c(16\beta \sigma - 6c(1-\beta c)^{2})}{3}}E=\\
       \\
       =-{\frac{8c\beta \sigma E}{3}}\left[ {\frac{(1-\kappa^{2})K}{(2-\kappa^{2})E}}-2+{\frac{6c(1-\beta c)^{2}}{8c\beta}}\right]<0
       \end{array}
      $$
   and $det(d^{''})<0$.

   (3) If $\beta>0$ and $8\beta \sigma-3c(1-\beta c)^2\leq 0$, then all terms of $B(c, \omega, \kappa, \beta)$
 are negatives and $det(d'')<0$.

   Thus under above three conditions, $d''(\omega , c)$ has exactly one positive and
one negative eigenvalues and $p(d'')=1$. This finishes the proof
of the Theorem.
\end{proof}




\smallskip

\end{document}